\documentclass[11pt]{iopart}

\usepackage{iopams}  

\usepackage{soul,color,xcolor}

\usepackage{bm,multirow}
\usepackage[ruled]{algorithm2e}

\usepackage{amsfonts}
\usepackage{graphicx}
\newcommand{\eqref}[1]{(\ref{#1})}

\usepackage[letterpaper,centering, top=1.5in, bottom=1.5in, left=1.2in, right=1.2in]{geometry}

\let\^=\hat
\let\-=\mathbf
\def\~#1{{\mbox{\sf#1}}}

\def\@G{{\mathcal{G}}}
\def\dis{{\mathcal{D}}}

\newcommand{\thetab}{\boldsymbol{\theta}}

\def\R{{\mathbb R}}

\def\N{{\mathbb N}}

\def\bi{{\mathbf i}}

\def\bi{\begin{itemize}} \def\ei{\end{itemize}}
\def\be{\begin{eqnarray*}}
\def\ee{\end{eqnarray*}}

\def\0{{\mathbf 0}}

\newcommand{\beq}{\begin{equation}}
\newcommand{\eeq}{\end{equation}}

\def\bth{\boldsymbol{\theta}}

\def\XXint#1#2#3{{\setbox0=\hbox{$#1{#2#3}{\int}$ }
\vcenter{\hbox{$#2#3$ }}\kern-.55\wd0}}

\newcommand{\Fc}{{\cal F}}

\newcommand{\Mc}{{\cal M}}
\newcommand{\Nc}{{\cal N}}

\def\zal2{z_{\alpha/2}}

\newcommand{\Gy}{\Gamma_{\rm y}}

\newcommand{\Gobs}{\Gamma_{\rm obs}}
\newcommand{\Gpr}{\Gamma_{\rm pr}}
\newcommand{\Gauss}{\Nc}

\newcommand{\Gprs}{S_{\rm pr}} %
\newcommand{\Gprst}{S_{\rm pr}^\top}

\newcommand{\Gpos}{\Gamma_{\rm pos}}

\newcommand{\Gposh}{\widehat{\Gamma}_{\rm pos}}

\newcommand{\mupos}{\bm\mu_{\rm pos}}

\newcommand{\df}{d_{\Fc}}
\newcommand{\trace}{\mathrm{tr}}

\newtheorem{remark}{Remark}[section]
\newtheorem{theorem}{Theorem}[section]
\newtheorem{lemma}[theorem]{Lemma}

\newtheorem{corollary}[theorem]{Corollary}
\newenvironment{proof}[1][Proof]{\begin{trivlist}
\item[\hskip \labelsep {\bfseries #1}]}{\end{trivlist}}

\begin{document}

\title[Approximate empirical Bayesian method for linear inverse
problems]{An approximate empirical Bayesian method for large-scale
  linear-Gaussian inverse problems}

\author{Qingping Zhou}

\address{Department of Mathematics and Institute of Natural Sciences, Shanghai Jiao Tong University, Shanghai 200240, China}
\ead{zhou2015@sjtu.edu.cn}

\author{Wenqing Liu}

\address{Department of Mathematics and Zhiyuan College, Shanghai Jiao Tong University, Shanghai 200240, China}
\ead{wenqingliu@sjtu.edu.cn}

\author{Jinglai Li}

\address{Institute of Natural Sciences, Department of Mathematics and MOE Key Laboratory of Scientific and Engineering Computing, Shanghai Jiao Tong University, Shanghai 200240, China}
\ead{jinglaili@sjtu.edu.cn}

\author{Youssef M.\ Marzouk}

\address{Department of Aeronautics and Astronautics, Massachusetts Institute of Technology, Cambridge, MA 02139,
USA}
\ead{ymarz@mit.edu}

\vspace{10pt}
\begin{indented}
\item[]March 2018
\end{indented}

\begin{abstract}
We study Bayesian inference methods for solving linear inverse
problems, focusing on hierarchical formulations where the prior or the likelihood function depend on unspecified hyperparameters. 
In practice, these hyperparameters are often determined via an empirical
Bayesian method that 
maximizes the marginal likelihood function, i.e., the probability density of the data conditional on the hyperparameters.
Evaluating the marginal likelihood, however, is computationally challenging for large-scale problems. 
In this work, we present a method to approximately evaluate marginal likelihood functions,  
 based on a low-rank approximation of the update from the prior
 covariance  to the posterior covariance.
We show that this approximation is optimal in a minimax sense. Moreover, we provide an efficient algorithm to implement
the proposed method, based on a combination of the randomized SVD and
a spectral approximation method to compute square roots of the prior
covariance matrix. 
Several numerical examples demonstrate good performance of the
proposed method.

\end{abstract}

%
%
%
%
%

\section{Introduction}
\label{sec:intro}
Bayesian inference approaches have become increasingly popular as a tool to solve inverse
problems~\cite{tarantola2005inverse,kaipio2005statistical,stuart2010inverse}.  In this
setting, the parameters of interest are treated as random variables, endowed with a
prior distribution that encodes information available before the data are
observed. Observations are modeled by their joint probability distribution conditioned on
the parameters of interest, which defines the likelihood function; this distribution
incorporates the forward model and a stochastic description of measurement or model
errors. The posterior distribution of the unknown parameters is then obtained via Bayes'
formula, which combines the information provided by the prior and the data. Here we focus
on a particular class of Bayesian inverse problems, where the forward model is linear and
both the prior and the error distributions are Gaussian. Such \textit{linear--Gaussian}
inverse problems arise in many real-world
applications~\cite{tarantola2005inverse,kaipio2005statistical}.

In practice, a difficulty in implementing Bayesian inference for such
problems is that the prior and/or the likelihood function may contain
unspecified hyperparameters. For instance, the prior distribution
usually encodes some kind of correlation among the inversion
parameters, but the correlation length is often
unknown. Alternatively, the variance and/or the mean of the
observational errors may not be available, thus yielding unspecified
hyperparameters in the likelihood function.
A natural idea to address such problems is to construct a \textit{hierarchical}
Bayesian model for both the inversion parameters and the
hyperparameters. Then both sets of parameters can be inferred from data
through the characterization of their {joint} posterior
distribution.  This inferential approach is often called ``fully
Bayesian.'' A fully Bayesian method can be prohibitively expensive for
large-scale inverse problems, however, as it requires sampling from a 
posterior distribution that is higher dimensional and no longer
Gaussian.  An alternative strategy is the empirical Bayes (EB)
approach~\cite{casella1985introduction,carlin2000bayes}, which first
estimates the hyperparameters by maximizing their marginal likelihood
function, i.e., the probability density of the data conditioned only on the
hyperparameters, and then plugs in the estimated values to compute the
posterior of the inversion parameters.
The use of the EB method has been theoretically
justified~\cite{petrone2014bayes,rousseau2015asymptotic}: roughly
speaking, these theoretical results show that, under certain
conditions and for a sufficiently large sample size, the EB method
leads to similar inferential conclusions as fully Bayesian inference.
At the same time, the EB method is more computationally efficient
than the fully Bayesian approach, as it does not require characterizing
the joint posterior distribution.

Nonetheless, the numerical implementation of EB methods remains
computationally taxing when dimensionality of the inversion
parameters is high, as it involves maximizing the marginal likelihood
function, which in turn requires ``integrating
out'' the inversion parameters for each evaluation. (Details about the
computational cost of evaluating the marginal likelihood function are
given in Section~\ref{sec:setup}.)  The goal of this work is to
present an \textit{approximate EB method} that can significantly
reduce the computational cost of solving linear inverse problems with
hyperparameters. Specifically, our method evaluates the
marginal likelihood function by using a \textit{low-rank update}
approximation of the posterior covariance matrix, introduced in previous work for the
non-hierarchical setting, e.g., \cite{flath2011fast,bui2013computational,martin2012stochastic,alexanderian2016bayesian,spantini2015optimal,spantini2016goal}.
The intuition behind the low-rank update approximation is that the
data may be informative, relative to the prior, only on a
\textit{low-dimensional subspace} of the entire parameter space.
%
%
As a result, one can consider approximations of the posterior
covariance matrix in the form of low-rank negative semidefinite updates of
the prior covariance matrix, where the update is obtained by solving a
generalized eigenproblem involving the log-likelihood Hessian and the
prior precision matrix. The optimality properties of this method are
studied in~\cite{spantini2015optimal}, which shows that this low-rank
approximation is {optimal} with respect to a specific class of
loss functions.

Using the approximation of the posterior covariance developed in
\cite{flath2011fast, bui2013computational, spantini2015optimal}, we introduce a new
method to efficiently compute the marginal likelihood function.
We prove that the proposed
method yields an \emph{optimal} approximation of the marginal
likelihood function in a minimax sense.
Our last contribution lies in the numerical implementation.  Unlike
the inverse problems considered in the work mentioned above,
where the prior is fixed, our problems require repeatedly
computing the square roots of different prior covariance matrices---which
can be prohibitively expensive when the dimensionality of the problem
is high.  To address the issue, we use the spectral approximation
developed in \cite{jiang2013fast} to compute the square root of the
prior covariance matrix, resulting in a very efficient algorithm for
evaluating (and maximizing) the marginal likelihood function.

%
The rest of this paper is organized as follows. In Section
\ref{sec:setup} we introduce the basic setup of the empirical Bayesian
method for solving inverse problems.  The low-rank update
approximation for evaluating the log marginal likelihood function is
developed and analyzed in Section \ref{sec:lr}, and the detailed
numerical implementation of the proposed method is discussed in
Section~\ref{sec:algorithm}.  Section \ref{sec:examples} provides two
numerical examples to demonstrate the performance of the proposed
method, and Section \ref{sec:conclusions} offers concluding remarks.

\section{Problem setup}\label{sec:setup}

Consider a linear inverse problem in a standard setting:
\[ \-y=G\-x + \bm{\eta}\]
where $\-y\in \R^m$ is the data, $\-x\in \R^n$ is the unknown, $G$ is a $n\times m$ matrix, often called the forward operator, 
and $\bm{\eta}$ is the observation noise. 
Such an inverse problem can be solved with a Bayesian inference
method: namely we assume that the prior probability density of $\-x$
is $\pi(\-x)$ and the likelihood function is $\pi(\-y|\-x)$, and thus the posterior is given by Bayes' theorem:
\[
\pi(\-x|\-y) = \frac{\pi(\-y|\-x)\pi(\-x)}{\pi(\-y)}.
\]
{ Throughout this paper, we use the notation $\pi(\cdot)$
  as a generic representation of probability density; the specific
  density is made clear by its arguments.}
We further assume a Gaussian likelihood and a Gaussian
prior with a non-singular covariance matrix $\Gpr \succ 0$  and, without loss of generality,
zero mean:
\begin{equation}\label{eq:GBLM}
\-y\mid \-x \sim \Gauss(G\-x,\,\Gobs),\quad \-x\sim \Gauss(\-0,\,\Gpr).
\end{equation}
In this setting, the resulting posterior distribution is also Gaussian
\begin{equation*}
\-x \mid  \-y \sim \Gauss(\mupos, \Gpos),
\end{equation*}
with mean and covariance matrix given by
\begin{equation}\label{eq:postmoms}
\mupos = \Gpos \,G^\top \Gobs^{-1}\,\-y\quad \mbox{and} \quad \Gpos = \left(H + \Gpr^{-1}\right)^{-1},
\end{equation}
where 
\begin{equation}\label{eq:hessian}
H= G^\top\Gobs^{-1}G
\end{equation}
is the Hessian of the log-likelihood.

Now we consider a more complex situation, where the matrices $G$,
$\Gpr$ and $\Gobs$ (or some of them) depend on a vector of unspecified
hyperparameters $\thetab \in \R^p$.
As mentioned in the introduction, a popular method to deal with such problems is the empirical Bayesian (EB) approach, which determines 
$\thetab$ by maximizing the marginal likelihood function:
\begin{equation}
\max_{\thetab} \pi(\-y|\thetab) =
\int\pi(\-y|\-x,\thetab)\pi(\-x|\thetab) d\-x 
\label{e:em}
\end{equation}
or the marginal posterior density:
\[
\max_{\thetab} \pi(\thetab|\-y) = p(\thetab)\int\pi(\-y|\-x,\thetab)\pi(\-x|\thetab)  d\-x 
\]
where $p(\thetab)$ is the prior density of $\thetab$. 
Note that the computations of these two objectives are very similar and so we shall only discuss \eqref{e:em}. 
It is easy to see that the optimization problem in \eqref{e:em} 
 is equivalent to minimizing the negative log marginal
 likelihood:
\begin{eqnarray}
{\min_{\thetab}  } -\log \pi(\-y|\thetab)\nonumber\\
=\frac12\-y^T\Gobs^{-1}\-y+\frac12\log|\Gobs|-\frac12 \-z^T\Gpos\-z
+\frac12\log\frac{|\Gpr|}{|\Gpos|}, \label{e:eb1}
\end{eqnarray}
where $G$, $\Gpr$ and $\Gobs$ depend on $\thetab$ and $\-z= G^T\Gobs^{-1}\-y$.
Note that an important special case arises when only the prior depends
on the hyperparameter $\thetab$ and the likelihood function is fixed;
in this case, we can simply minimize  
\begin{equation}
  L(\thetab,\-z):= -\frac12 \-z^T\Gpos\-z
+\frac12\log\frac{|\Gpr|}{|\Gpos|}. \label{e:minL}
\end{equation}
For conciseness, below we will present our method for \eqref{e:minL}, 
 while noting that all the results can be trivially extended to \eqref{e:eb1}.

Direct evaluation of \eqref{e:minL} is not desirable for large scale problems, 
as it requires several operations with $O(n^3)$ complexity. 
In what follows, we present an accurate and  efficient---with $O(n^2 r)$ complexity for some $r\ll n$---method to approximate $L(\thetab)$, based on
a rank-$r$ update approximation of $\Gpos$.

\section{The optimal approximation method} \label{sec:lr}

\subsection{Optimal low-rank update approximation of the posterior covariance}
The proposed method begins with
 the {optimal} low-rank update approximation of posterior covariance developed in \cite{bui2013computational,spantini2015optimal},
 which is briefly reviewed here. 
 Note that $\Gpos$ can be written as a non-positive update of $\Gpr$:
$$ 
\Gpos = \Gpr - KK^\top,
$$
where
$$
KK^\top = \Gpr \,G^\top\Gy^{-1} G\, \Gpr
$$ 
and $\Gy = \Gobs + G\,\Gpr \,G^\top$ is the covariance of the marginal
distribution of $\-y$. This update of $\Gpr$ is negative semidefinite
because the data add information; they cannot increase the prior
variance in any direction. 
As has been discussed in previous
work~\cite{spantini2015optimal}, in many practical problems,  the low-rank structure often lies in 
the update of $\Gpr$ that yields $\Gpos$, rather than  in $\Gpos$
itself. Hence, \cite{spantini2015optimal} proposes to use the following class of positive definite matrices to approximate $\Gpos$:
\begin{equation} \label{eq:class_M}
    {\Mc}_r=\left\{ \Gpr-BB^\top\succ0  :\mathrm{rank}(B)\le r  \right\}.
\end{equation}

To establish optimality statements regarding the approximation of a covariance matrix, 
\cite{spantini2015optimal} adopts the following distance between
symmetric positive definite (SPD) matrices of size $n$:
\[
\df^2(A,B)=\trace \left [ \,\ln^2 (\,A^{-1/2}BA^{-1/2}\,)\, \right ]=\sum_{i=1}^n \ln^2(\sigma_i), 
\]
where $(\sigma_i)_{i=1}^n$ are the generalized eigenvalues of the pencil
$(A,B)$. 
This metric was first introduced by Rao in \cite{rao1945information}, and can measure the difference between two Gaussian distributions with the same mean. 
We direct interested readers to \cite{spantini2016goal} for a detailed discussion and other applications of  this metric. 
It is important to note that this distance is generalized in \cite{spantini2015optimal} to be
\[
\dis(A,B)=\sum_{i=1}^n f(\sigma_i), 
\]
where $f$ is a function in $C^1(\R^+)$ that satisfies $f'(x)(1-x)<0$ for any $x\neq1$, and $\lim_{x\rightarrow\infty} f(x) = \infty$.
This generalization will be used for the proof of our optimality statement in next section. 
Thus we seek a low-rank approximation $\Gposh$ of the covariance matrix $\Gpos$,   such that 
\begin{equation}
\Gposh = \arg \min_{\Gamma \in \Mc}\dis(\Gpos,\Gamma).\label{e:mindis}
\end{equation} 
The solution of \eqref{e:mindis} can be derived analytically regardless of the specific choice of $f$.
Specifically, let $\Gprs$ be any square root of
  the prior covariance matrix such that
  $\Gpr=\Gprs\,\Gprst$. We define the prior-preconditioned Hessian as
\begin{equation}\label{eq:precondH}
\widehat{H}=\Gprst H\,\Gprs,
\end{equation}
which plays an essential role in our numerical method. 
Now let $(\delta_i^2,\-w_i)$ be the
  eigenvalue-eigenvector pairs of $\widehat{H}$ 
with the ordering $\delta_i^2\ge\delta_{i+1}^2$;
then a solution of \eqref{e:mindis} is given by:
  \begin{equation} \label{e:minimizer_theorem}
    \Gposh=\Gpr-BB^\top, \quad BB^\top=\sum_{i=1}^r \delta_i^2 \left( 1 + \delta_i^2 \right)^{-1} \widehat{\-w}_i \widehat{\-w}_i^\top,
    \quad \widehat{\-w}_i=\Gprs \-w_i .
  \end{equation}
  The corresponding minimum  distance is 
  \[ 
  \dis(\Gpos,\Gposh) = f(1)r+f{\sum_{i=r+1}^n  f(1/(1+\delta_i^2))}.
  \] 
The minimizer \eqref{e:minimizer_theorem} is unique if the
  first $r$ eigenvalues of $\widehat{H}$ are distinct.

\subsection{Approximating the log-likelihood function}
Now we apply the optimal low rank approximation to our problem. 
The idea is rather straightforward: we approximate $L$ in  \eqref{e:minL} with 
\begin{equation}
\widehat{L}(\thetab,\-z):=-\frac12 \-z^\top\,\Gposh\,\-z
+\frac12\log\frac{|\Gpr|}{|\Gposh|}, \label{e:Lhat}
\end{equation}
for some approximate posterior covariance matrix $\Gposh$. 
Next we shall derive the appropriate choice of $\Gposh$. 
To do this, 
 we need to impose an additional assumption on the approximate posterior covariance matrix: $\Gposh-\Gpos \succeq0$, which means that the approximation itself should not create any new information.
 As a result, the class of matrices for approximating $\Gpos$ becomes
\begin{equation} \label{eq:class_M2}
    {\Mc}_r'=\left\{\Gposh= (\Gpr-BB^\top) : \Gposh-\Gpos\succeq0 ,\, \mathrm{rank}(B)\le r  \right\}
\end{equation}
for some maximum rank $r$. 
Next we shall consider the approximations of the two terms $\log({|\Gpr|}/{|\Gpos|})$ and  $\-z^\top\,\Gpos\,\-z$ in \eqref{e:minL} separately.

\subsubsection{Approximating the log--determinant term.}
First, we consider finding a matrix $\Gposh\in \Mc_r'$ to approximate
$\log ({\vert \Gpr \vert}/{ \vert \Gpos \vert})$ with 
 $\log ({|\Gpr|}/{|\Gposh|})$. 
In this setting, it is easy to see that the approximation error in the
log-marginal likelihood \eqref{e:minL} is  
$\vert \log({|\Gpos|}/{|\Gposh|}) \vert $, 
and  a natural way to determine $\Gposh$ is to seek a $\Gposh\in \Mc_r'$ that minimizes this approximation error. 
To this end, we have the following theorem:
\begin{theorem}\label{th:err1}
Suppose that we approximate  $\log ({|\Gpr|}/{|\Gpos|})$ with 
 $\log ({|\Gpr|}/{|\Gposh|})$ for some  $\Gposh\in\Mc_r'$.  
The matrix $\Gposh\in\Mc'_r$ that minimizes the resulting approximation error, i.e., the solution of 
\begin{equation} 
\min_{\Gposh\in\Mc_r'}   \left \vert \log\frac{|\Gpos|}{|\Gposh|}
\right \vert , \label{e:minerr}
\end{equation}
is given by \eqref{e:minimizer_theorem}. 
Moreover, the optimal approximation and the associated error are, respectively, 
\begin{equation}
\log\frac{|\Gpr|}{|\Gposh|} = \sum_{i=1}^r \log(1+\delta_i^2) \ \mathit{and} \ \log\frac{|\Gposh|}{|\Gpos|} = \sum_{i=r+1}^n \log(1+\delta_i^2). \label{e:optval}
\end{equation}
\end{theorem}
\begin{proof}
We prove this theorem using the optimality results in \cite{spantini2015optimal}.
To start, we choose a particular distance metric by letting
 \[f(x) =  \vert \log x \vert .\]
 We denote the resulting distance metric as $\dis_1(A,B)$ to indicate
 that such a metric is actually the $1$-norm of
 $(\ln(\sigma_1),\,\ldots,\,\ln(\sigma_n))$ while $\df$ is the $2$-norm
 of the same sequence.
It can be verified that 
\[
\log\frac{|\Gpr|}{|\Gpos|} = \sum_{i=1}^n \log(1+\delta_i^2) = \dis_1(\Gpr,\Gpos),
\]
as $\log(1+\delta_i^2)\geq0$ for all $i\in\N$.
Since the approximate posterior covariance $\Gposh \in \Mc_r'$, we can show
\begin{equation}
\dis_1(\Gpr,\Gpos) = \dis_1(\Gpr,\Gposh)+\dis_1(\Gposh,\Gpos), \label{e:D1_sep}
\end{equation}
where $\dis_1(\Gpr,\Gposh) = \log({|\Gpr|}/{|\Gposh|})$ is the approximation and $\dis_1(\Gposh,\Gpos) =\log({|\Gposh|}/{|\Gpos|})$ is the error associated with it. 
Note that \eqref{e:D1_sep} does not hold without the assumption $\Gposh\in\Mc_r'$.
Thus \eqref{e:minerr} can be rewritten as, 
\begin{equation}
\min_{\Gposh\in \Mc_r'}\dis_1(\Gpos,\Gposh).\label{e:mindis1}
\end{equation}
Recall that the solution of $\min_{\Gposh\in \Mc_r}\dis_1(\Gpos,\Gposh)$
 is given by \eqref{e:minimizer_theorem},
and it is easy to verify that the matrix $\Gposh$ given by \eqref{e:minimizer_theorem} is in $\Mc_r'$,
which implies that \eqref{e:minimizer_theorem}  also provides the solution of \eqref{e:mindis1}.
As a result, the optimal approximation and the associated approximation error are given by \eqref{e:optval},
which completes the proof. 
\end{proof}

\subsubsection{Approximating the quadratic term.}
Similarly, we can also find an approximate posterior covariance $\Gposh\in\Mc_r'$ and approximate $\-z^\top\,\Gpos\,\-z$ with $\-z^\top\,\Gposh\,\-z$. 
This problem is a bit more complicated: 
since $\-z$ (which is a linear transformation of the data $\-y$) is random, we cannot determine the matrix $\Gposh$ by directly minimizing the approximation error. In this case,  a useful treatment is to apply the minimax criterion, i.e., to seek a matrix $\Gposh\in\Mc'_r$ that minimizes the maximum approximation 
error with respect to $\-z$. In particular, for the maximum error to exist, we shall require $\-z$ to be bounded:
   $\-z\in Z_c = \{\-z : \|\-z\|_2 \leq c \}$ for a constant $c>0$. (See Remark~\ref{rem:aboutc} for a discussion of this boundedness assumption.)
The following theorem provides the optimal solution to this problem. 
\begin{theorem}\label{th:err2}
Suppose that we approximate  $\-z^\top\,\Gpos\,\-z$ with 
 $\-z^\top\,\Gposh\,\-z$, for some  $\Gposh\in\Mc_r'$.   
The matrix $\Gposh\in\Mc_r'$ that achieves the minimax approximation error, i.e., the solution of 
\begin{equation} 
\min_{\Gposh\in\Mc_r'} \max_{\-z\in Z_c} |\-z^\top\,\Gpos\,\-z-\-z^\top\,\Gposh\,\-z| \label{e:minerr2}
\end{equation}
is given by \eqref{e:minimizer_theorem}. 
Moreover, the resulting approximation is
\begin{equation}
\-z^\top\, \Gposh \,\-z=\-z^\top\,\Gpr\,\-z- \-z^\top\, \-b, \quad \-b=\sum_{i=1}^r \frac{\delta^2_i}{{ 1 + \delta_i^2 }} (\widehat{\-w}_i^\top\-z)\widehat{\-w}_i, \label{e:optval2}
\end{equation}
and the associated approximation error is
\begin{equation}
 |\-z^\top \,(\Gposh-\Gpos) \,\-z| = \-z^\top \left(\sum_{i=r+1}^n \frac{\delta^2_i}{{ 1 + \delta_i^2 }} (\widehat{\-w}_i^\top\-z)\widehat{\-w}_i \right).
 \label{e:err3}
\end{equation}
\end{theorem}
\begin{proof}
For any given $\Gposh\in \Mc_r'$,  it is easy to see that the solution of 
\[\max_{\-z\in Z_c}|\-z^\top\,\Gpos\,\-z-\-z^\top\,\Gposh\,\-z| = \-z^\top\,(\Gposh-\Gpos)\,\-z\]
 is $\-z=c {\-v}_{\max}$ where $\-v_{\max}$ is the eigenvector of the largest eigenvalue of $\Gposh-\Gpos$, and the maximum error is 
 $c^2\rho(\Gposh-\Gpos) $ where $\rho(\cdot)$ is the spectral
 radius. Thus \eqref{e:minerr2} becomes 
$\min_{\Gposh\in\Mc_r'}\rho(\Gposh-\Gpos)$, or equivalently, 
$$\min_{\mathrm{rank}(B)\leq r} \rho(KK^\top-BB^\top),$$
and it follows immediately that the optimal $\Gposh$ is given by \eqref{e:minimizer_theorem}.
Substituting \eqref{e:minimizer_theorem} into $\-z^\top\,\Gposh\,\-z$ and 
 $\-z^\top\,(\Gposh -\Gpos )\,\-z$ yields \eqref{e:optval2} and \eqref{e:err3}, respectively. 
\end{proof}

{In principle, of course, $\|\bm{z}\|_2$ is not bounded from above, since $\bm{z}$ is normally distributed.
However, imposing the boundedness assumption, which considerably simplifies the theoretical analysis,  does not
limit the applicability of the method or affect the optimal solution. 
Specifically, we have the following remarks:}
\begin{remark}\label{rem:aboutc}

{ First, since $\bm{z}$ follows a Gaussian distribution, one can always choose a constant $c$ such that the inequality $ \|\bm{z}\|_2 \leq c$ holds with probability arbitrarily close to one.
Second, 
the minimax solution $\Gposh$ is \textit{independent} of the value of $c$.}
\end{remark}


Now we combine the two approximate treatments, which is essentially to approximate \eqref{e:minL} by \eqref{e:Lhat} with
$\Gposh$ given by \eqref{e:minimizer_theorem}. 
It is easy to see that the approximation error is 
\begin{equation}
\Delta L(\thetab,\-z) = \frac12\-z^\top\,(\Gposh-\Gpos)\,\-z+\frac12\log({|\Gposh|}/{|\Gpos|}), \label{e:DL}
\end{equation}
and we have the following result regarding its optimality: 
\begin{corollary}
Suppose that we approximate \eqref{e:minL} with \eqref{e:Lhat} for some matrix $\Gposh\in \Mc_r'$. 
 The matrix $\Gposh$ given by \eqref{e:minimizer_theorem} achieves the minimax approximation error, i.e., it solves
$$\min_{\Gposh\in\Mc_r'} \max_{\-z\in Z_c} |\Delta L(\thetab,\-z) |.$$ 
\end{corollary}
As both terms on the right hand side of \eqref{e:DL} are nonnegative, 
the corollary  is a direct consequence of Theorems~\ref{th:err1} and \ref{th:err2}.
 
\section{Numerical implementation}\label{sec:algorithm}
Here we discuss the numerical implementation of our approximate method to evaluate \eqref{e:minL}. In principle, this involves two computationally intensive components---both requiring $O(n^3)$ computations under standard numerical treatments. The first task is to compute the eigenvalues of $\widehat{H}$.
Recall that our method only requires the $r$ leading eigenvalues and associated eigenvectors of $\widehat{H}$, which can be computed efficiently with 
a randomized algorithm for the singular value decomposition (SVD)~\cite{liberty2007randomized,halko2011finding}. 
The second task is to compute $\Gprs$, the square root of $\Gpr$.
As will become clear later, we do not necessarily need $\Gprs$; rather, our algorithm only requires the ability to evaluate $\Gprs \Omega$ for a given 
matrix $\Omega$.   
To this end, we resort to the approximation method proposed in \cite{jiang2013fast}. 
We provide a brief description of  the two adopted methods here, tailored according to our specific purposes. 

\subsection{Randomized SVD}\label{sec:rsvd}
The main idea behind the randomized SVD approach is to identify a subspace that captures most of the action of a matrix, using random sampling~\cite{liberty2007randomized,halko2011finding,drineas2006fast}.
 The original matrix is then projected onto this subspace, yielding a
 smaller matrix, and a standard SVD is then applied to the smaller
 matrix to obtain the leading eigenvalues of the original matrix.  A
 randomized algorithm for constructing the rank--$r'$ SVD of an
 $n\times n$ matrix $\widehat{H}$, given in \cite{halko2011finding}, proceeds as follows:
\begin{enumerate}
\item  Draw an $n\times r'$ Gaussian random matrix $\Omega$;
\item Form the $n\times r'$ sample matrix $Y=\widehat{H} \Omega$;
\item Form an $n\times r'$ matrix $Q$ { with orthonormal columns}, such that $Y=Q R$;
\item Form the $r'\times n$ matrix $B = Q^T \widehat{H}$;
\item Compute the SVD of the $r'\times n$ matrix $B= \hat{U} \Sigma V^T$.
\end{enumerate}
According to \cite{halko2011finding}, for this algorithm to be robust, it is important to oversample a bit, namely to choose $r'>r$ if $r$ is the desired rank. 
{  One can obtain a probabilistic error bound, i.e., that
\begin{equation} \|\widehat{H}-QB\|_2 \leq (1+11\sqrt{r'n})\delta_{r+1}, \end{equation}
holds with a probability at least $1-6(r'-r)^{-(r'-r)}$,
under some very mild assumptions on $r'$~\cite{halko2011finding}.}
{ Note that the diagonal entries of $\Sigma$ are the eigenvalues of $\widehat{H}$ and the columns of $V$ are the eigenvectors. 
Since we obtain $r'$ eigenpairs in the algorithm, we take $\{\delta_i,\-w_i\}_{i=1}^r$ to be the $r$ dominant ones among them.}
Above we only present the basic implementation of the randomized SVD method; further details and possible improvements of the method can be found in \cite{halko2011finding} and the references therein.
{  Finally we also want to emphasize that our approximate EB
  method can use the eigenvalues/eigenvectors computed with any
  numerical approach; it is not tailored to or tied with the randomized SVD. }

\subsection{The Chebyshev spectral method for computing $\Gprs$} \label{sec:cheb}
Recall that $\widehat{H}=\Gprst H\,\Gprs$, and as a result the randomized SVD only requires evaluating ${\Gprs} \Omega$ where $\Omega$ is a randomly generated $n\times r'$ matrix. 
Here we adopt the approximate method proposed in \cite{jiang2013fast}, which is based on the following lemma:
\begin{lemma}[Lemma 2.1 in \cite{jiang2013fast}]
Suppose that $D$ is a real symmetric positive definite matrix. Then there exists a polynomial $p(\cdot)$ such that $\sqrt{D} = p(D)$, and the degree of $p$ is equal to the number of distinct eigenvalues of $D$ minus $1$.
\end{lemma}
Though $\sqrt{D}$ is exactly equal to a polynomial of $D$, the degree of the polynomial might be very large if $D$ has a large number of distinct eigenvalues. Thus, instead of trying to find the exact polynomial $p$ in $D$ that equals $\sqrt{D}$,  the aforementioned work computes a Chebyshev approximation to it. 

It is clear that $D=O\Lambda O^{T}$, where $\Lambda=\mathrm{diag}(\lambda_{1}\cdots \lambda_{n})$, and $\lambda_{1}\geq\cdots\geq\lambda_{n}$. 
Now suppose that we have a polynomial $p(\cdot)$ such that $p(\lambda) \approx\sqrt{\lambda}$. Then we have
\begin{eqnarray*}
p(D)&=&O \, p(\Lambda) \, O^{T}\\
&=& O \, \mathrm{diag}(p(\lambda_{1})\cdots p(\lambda_{n})) \, O^{T}\\
&\approx &O \, \mathrm{diag}(\sqrt{\lambda_{1}}\cdots \sqrt{\lambda_{n}}) \, O^{T}\\
&=& \sqrt{D}.
\end{eqnarray*}
That is, once we have an approximation of ${\sqrt{\lambda}}$ in the interval $[\lambda_{\min},\lambda_{\max}]$, where $\lambda_{\min}$ and $\lambda_{\max}$ are respectively lower and upper bounds on the eigenvalues of $D$, 
we immediately get an approximation of $\sqrt{D}$.
A popular method to construct the approximation $p(\cdot)$ is the Chebyshev polynomial interpolation. 
The standard Chebyshev polynomials are defined in $[-1,\,1]$ as~\cite{press2007numerical}, 
\begin{equation}
T_k(x) = \cos(k \arccos x),\quad \forall k\in\N, \label{e:cheb}
\end{equation}
and the associated Chebyshev points are given by
\[
x_j = \cos [\frac{(2j+1)\pi}{2k+2}], \quad j=0,1,\ldots,k.
\]
It is well known that the Chebyshev interpolant has spectral accuracy for analytical functions on $[-1,1]$; see, e.g. \cite{fornberg1998practical}. 

Here since we intend to approximate the function $\sqrt{\lambda}$ over the range $[\lambda_{\min},\lambda_{\min}]$ rather than $[-1,\,1]$, we shall use the scaled and shifted Chebyshev polynomials:
\begin{equation}
\tilde{T}_k(x) = T_k(t_ax+t_b),\quad \forall n\in\N,
\end{equation}
and the associated Chebyshev points become
\begin{equation}
x_{j}=\frac{\lambda_{\max}+\lambda_{\min}}{2}+
\frac{\lambda_{\max}-\lambda_{\min}}{2}
\cos(\frac{(2j+1)\pi}{2k+2}), \quad j=0,\,1,\,\ldots,\,k.
\end{equation}
The  interpolant $p_{k}(x)$ can be expressed as~\cite{press2007numerical},
\begin{equation}
p_{k}(x)= \sum_{i=0}^{k}c_{i}\tilde{T_{i}}(x) -\frac{c_0}2,
\end{equation}
where the coefficients $c_{i}$ are given by
\begin{equation}
c_{i}=\frac{2}{k+1}\sum_{j=1}^{k+1}\sqrt{x_{j}}
\tilde{T_{i}}(x_{j}).
\end{equation}
It is easy to verify that the scaled and shifted Chebyshev polynomials satisfy the following recurrence relation, 
\begin{eqnarray}
\tilde{T}_{0}&=&1,\nonumber\\
\tilde{T}_{1}(x)&=&t_{a}x+t_{b},\label{e:recur}\\
\tilde{T}_{i+1}(x)&=&2(t_{a}x+t_{b})\tilde{T}_{i}(x)
-\tilde{T}_{i-1}(x),\nonumber
\end{eqnarray}
with
\[
t_{a}=\frac{2}{\lambda_{\max}-\lambda_{\min}},\quad
t_{b}=\frac{\lambda_{\max}+\lambda_{\min}}
{\lambda_{\max}-\lambda_{\min}}.
\]
Now recall that we actually want to compute $\sqrt{D}\Omega$.
Taking advantage of the recurrence relation~\eqref{e:recur}, we obtain,
\begin{equation}
\sqrt{D}\Omega\approx p_k(D)\Omega=\sum_{i=0}^{k}c_{i}\tilde{T_{i}}(D)\Omega
-\frac{c_{0}\Omega}{2}=
\sum_{i=0}^{k}c_{i}\Omega_{i}-\frac{c_{0}}{2}\Omega_0,
\end{equation}
where 
\begin{eqnarray*}
\Omega_{0}&=&\Omega,\\
\Omega_{1}&=&t_{a}D\Omega_{0}+t_{b}\Omega_{0},\\
\Omega_{i+1}&=&2(t_{a}Dv_{i}+t_{b}\Omega_{i})-\Omega_{i-1}, 
\end{eqnarray*}
for $i=1, \ldots, k-1$. 
The complete procedure for computing $\sqrt{D}\Omega$ with the Chebyshev approximation is given in Algorithm~\ref{alg:sqrtD}.

\begin{algorithm}[t]\label{alg:sqrtD}
 \KwData{$\Omega$, $k$}
 \KwResult{$B \approx \sqrt{D}\Omega$ }
Compute the coefficients,  $t_a$, $t_b$, and  $c_0,\ldots,c_k$\;
$\Omega_0 := \Omega$\;
  $\Omega_1 := t_a D \Omega_0 + t_b \Omega_0$\;
 \For{$i=1$ to $k-1$}{
 $\Omega_{i+1} := 2(t_a D \Omega_i + t_b \Omega_i)-\Omega_{i-1}$\;
 }
 $B:= \sum_{i=0}^kc_i\Omega_i-\frac{c_{0}}{2}\Omega_0$.
 \caption{Chebyshev spectral approximation for computing $\sqrt{D}\Omega$}
\end{algorithm}

Finally we provide some remarks regarding the implementation of the method:
\begin{itemize}
\item The original algorithm present in \cite{jiang2013fast} is to compute the product of $\sqrt{D}$ and a vector, but as is shown in Algorithm~\ref{alg:sqrtD}, its extension to  
 the computation of the product of $\sqrt{D}$ and a matrix is straightforward.
\item The method requires upper and lower bounds on the eigenvalues of $D$, which are computed with an algorithm based on the safeguarded Lanczos method~\cite{jiang2013fast}. As is discussed in \cite{jiang2013fast}, these estimates need not be of high accuracy.  
\item The error bound of the proposed Chebyshev approximation is given by Theorems 3.3 and 3.4 in \cite{jiang2013fast}. 
\item If desired, all the matrix-vector multiplications can be performed with the fast multipole method to further improve efficiency~\cite{jiang2013fast}. 
\end{itemize}

{ \subsection{The complete algorithm}}
{ Now we summarize the complete scheme for
  constructing the approximation $\widehat{L}(\thetab,\-z)$ of \eqref{e:minL} with a given rank $r$, using the methods discussed in Sections~\ref{sec:rsvd} and \ref{sec:cheb}: 
\begin{enumerate}
\item Compute the first $r$ eigenpairs, $\{(\delta_i^2,\-w_i) \}_{i=1}^r$, of $\widehat{H}(\thetab)=\Gprs H\Gprs$, using the
  randomized SVD method and the Chebyshev spectral method.

\item Let $\widehat{\-w}_i= \Gprs \-w_i$ and evaluate \[\widehat{L}(\thetab, \-z) = \-z^\top\,\Gpr\,\-z- \-z^\top\, \sum_{i=1}^r \frac{\delta^2_i}{{ 1 + \delta_i^2 }} (\widehat{\-w}_i^\top\-z)\widehat{\-w}_i+\sum_{i=1}^r \log(1+\delta_i^2).\] 
\end{enumerate}
}

\section{Examples}\label{sec:examples}
\subsection{An image deblurring problem}
We first test our method on an imaging deblurring problem, which involves recovering a latent image from noisy observations of a blurred version of the image~\cite{hansen2006deblurring}. In particular, it is assumed that the blurred image is obtained as a convolution of the latent image with  a point spread function (PSF), 
and as a result the forward model is: 
\begin{equation}
y(t_1,t_2) =  \int \int_{ D} f_\mathrm{PSF}(t_1,t_2) x(\tau_1,\tau_2) d\tau_1d\tau_2,
\end{equation}
where $f_\mathrm{PSF}(t_1,t_2)$ is the PSF and $D$ is the domain of the image.
In this example we take  the image domain to be $D = [-1, 1]^2$ and the PSF to be 
\[f_\mathrm{PSF}(t_1,t_2)=\exp[-({(t_1-\tau_1)^2}+{(t_2-\tau_2)^2})/t],\] 
where $t$ is a parameter controlling the size of the spreading.
Moreover, we assume that the data  $y$ are measured at $m=64^2=4096$ observation locations evenly distributed in $D$,
and that the observation errors are mutually independent and Gaussian with zero mean and variance $(0.1)^2$. 
We represent the unknown $x$ on $256\times256$ mesh grid points, and thus the dimensionality of the inverse problem is $n=65536$. 
The prior on $x$ is a Gaussian distribution with zero mean and covariance kernel~\cite{stein2012interpolation,rasmussen2006gaussian}:
\begin{equation}
K(\-t,\-t') = \sigma^2 \frac{2^{1-\nu}}{\Gamma(\nu)} \left ( \sqrt{2\nu}\frac{d}{\rho} \right )^\nu B_\nu \left ( \sqrt{2\nu}\frac{d}{\rho} \right ), \label{e:matern}
\end{equation}
where $d=\|\-t-\-t'\|_2$, $\Gamma(\cdot)$ is the Gamma function, and $B_\nu(\cdot)$ is the modified Bessel function. 
\eqref{e:matern} is known as the Mat\'ern covariance, and several authors have suggested that such covariances can often provide better models for many realistic physical processes than the popular squared exponential covariance~\cite{stein2012interpolation,rasmussen2006gaussian}. 
A random function with the Mat\'ern covariance is $[\nu-1]$ mean-square (MS) differentiable. 
In this example, we choose $\nu=3$ implying second order MS differentiability. 
We set the standard deviation $\sigma$ to one. The correlation length $\rho$ is treated as a hyperparameter to be inferred in this example.

We now use the proposed empirical Bayesian method to solve the inverse problem. 
We first assume that the true correlation length is $\rho=1$, and randomly generate
a true image from the associated prior, shown in Fig.~\ref{f:data} (left).  
We then test two cases of the forward problem, with $t=0.002$ and $t = 0.02$, where the latter yields an inverse problem that is much more ill-posed than 
the former. 
We assume that the data are observed on a $128\times128$ uniformly distributed mesh.
We then apply the two convolution operators to the generated image and add observational noise to the results, producing the synthetic data also shown in Fig.~\ref{f:data}.

Using the proposed approximate EB method, we evaluate the negative log marginal likelihood function $L$ over a range of $\rho$ values, for both cases of $t$, and plot the results in Fig~\ref{f:L_rho}.
For $t=0.002$ (left figure), we compute $L$ with ranks $r=300$, $400$ and $500$, and $600$, and observe that the curves converge as the rank increases; indeed, the results with $r=500$ and $600$ appear identical. 
For $t=0.02$---because the problem is more ill-posed---we can implement the method with ranks $r=50$, $75$, $100$, and $150$ and observe convergence. In particular, while the results of $r=50$ deviate from the others, the results with $r=75$, $100$, and $150$ look nearly identical, implying that $r=75$ is sufficient for an accurate approximation of the marginal likelihood in this case. 
For both cases,  the optimal value of $\rho$ is found to be
$0.1$, which is actually the true hyperparameter value. 
We then compute the posterior mean of $x$, after fixing $\rho=0.1$ in the Gaussian prior on $x$,
and show the results in Fig.~\ref{f:postmean_conv}. 

\begin{figure}
\centerline{\includegraphics[width=.36\textwidth]{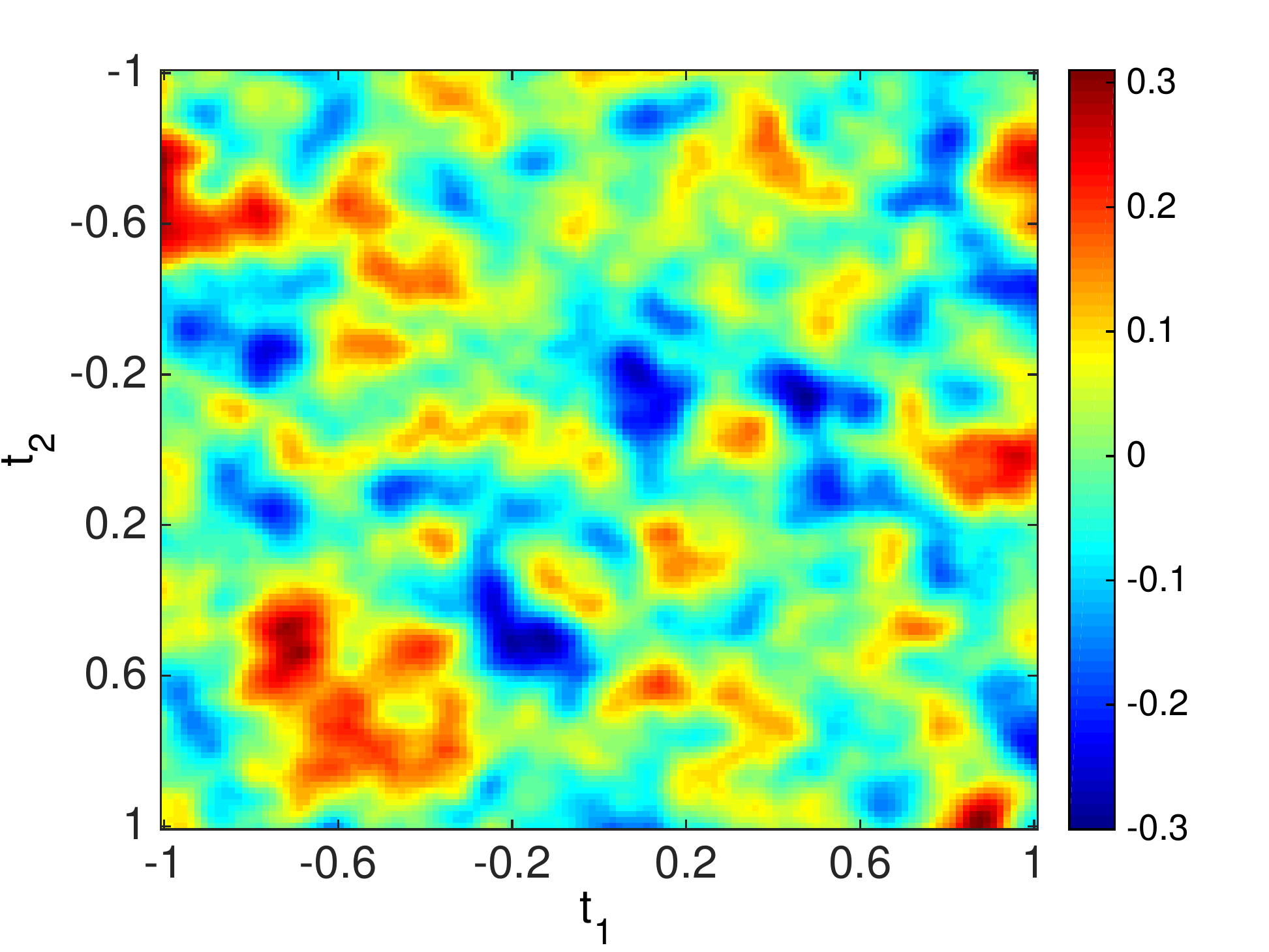}
\includegraphics[width=.36\textwidth]{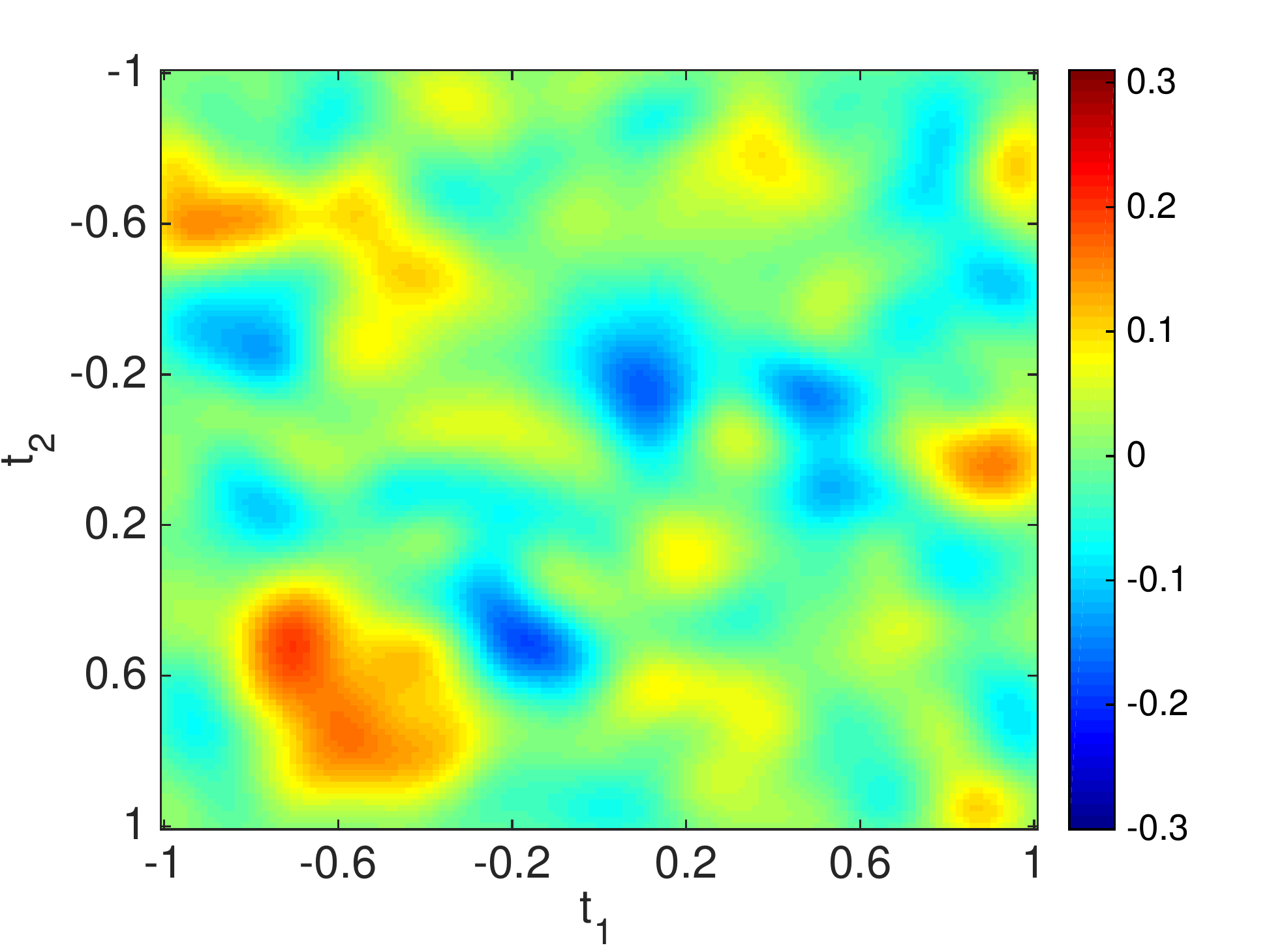}
\includegraphics[width=.36\textwidth]{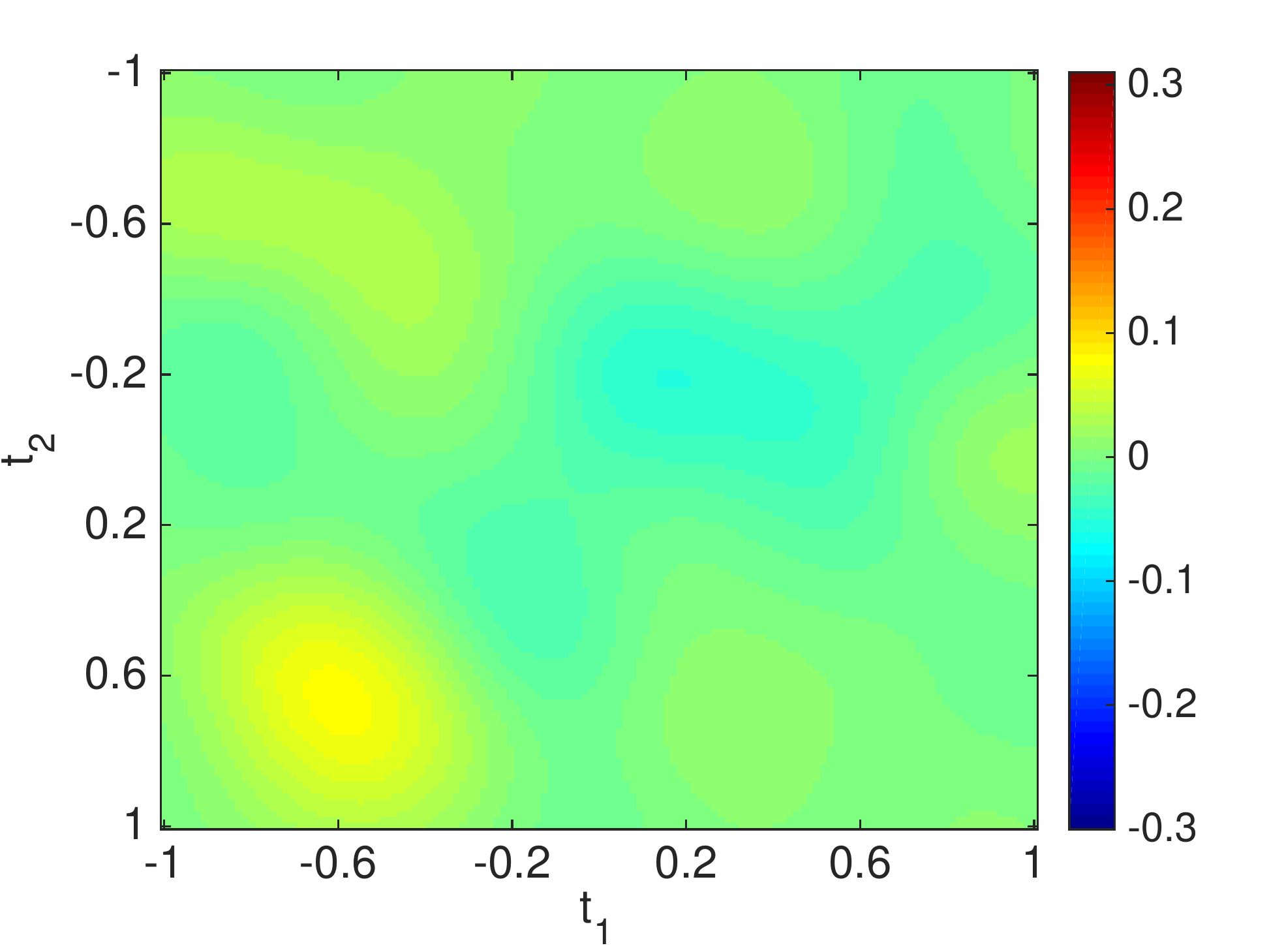}} 
\caption{From left to right: the ground truth, the simulated data at $t=0.002$ and that at $t=0.02$.}\label{f:data}
\end{figure}

\begin{figure}
\centerline{\includegraphics[width=.5\textwidth]{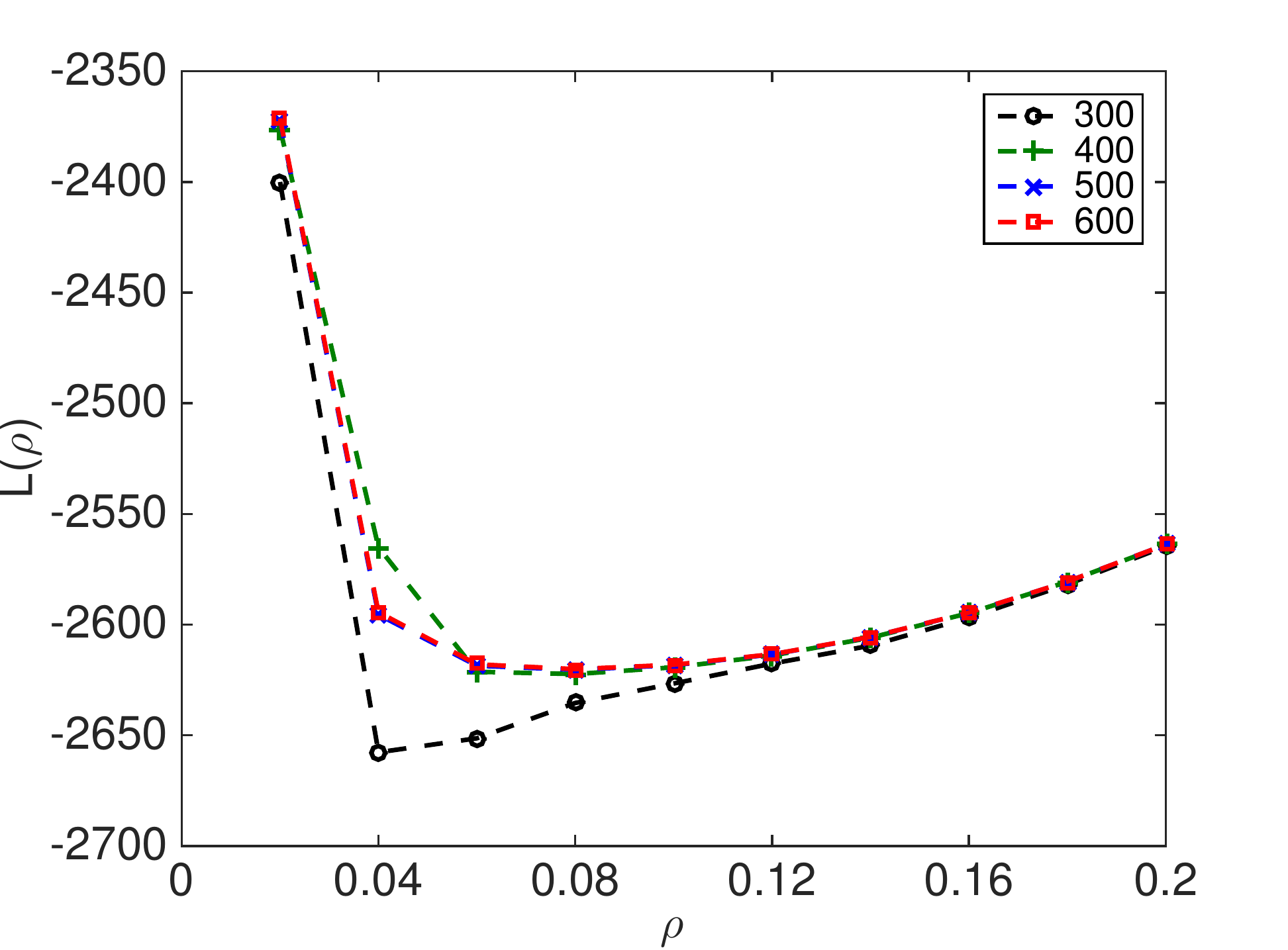}
\includegraphics[width=.5\textwidth]{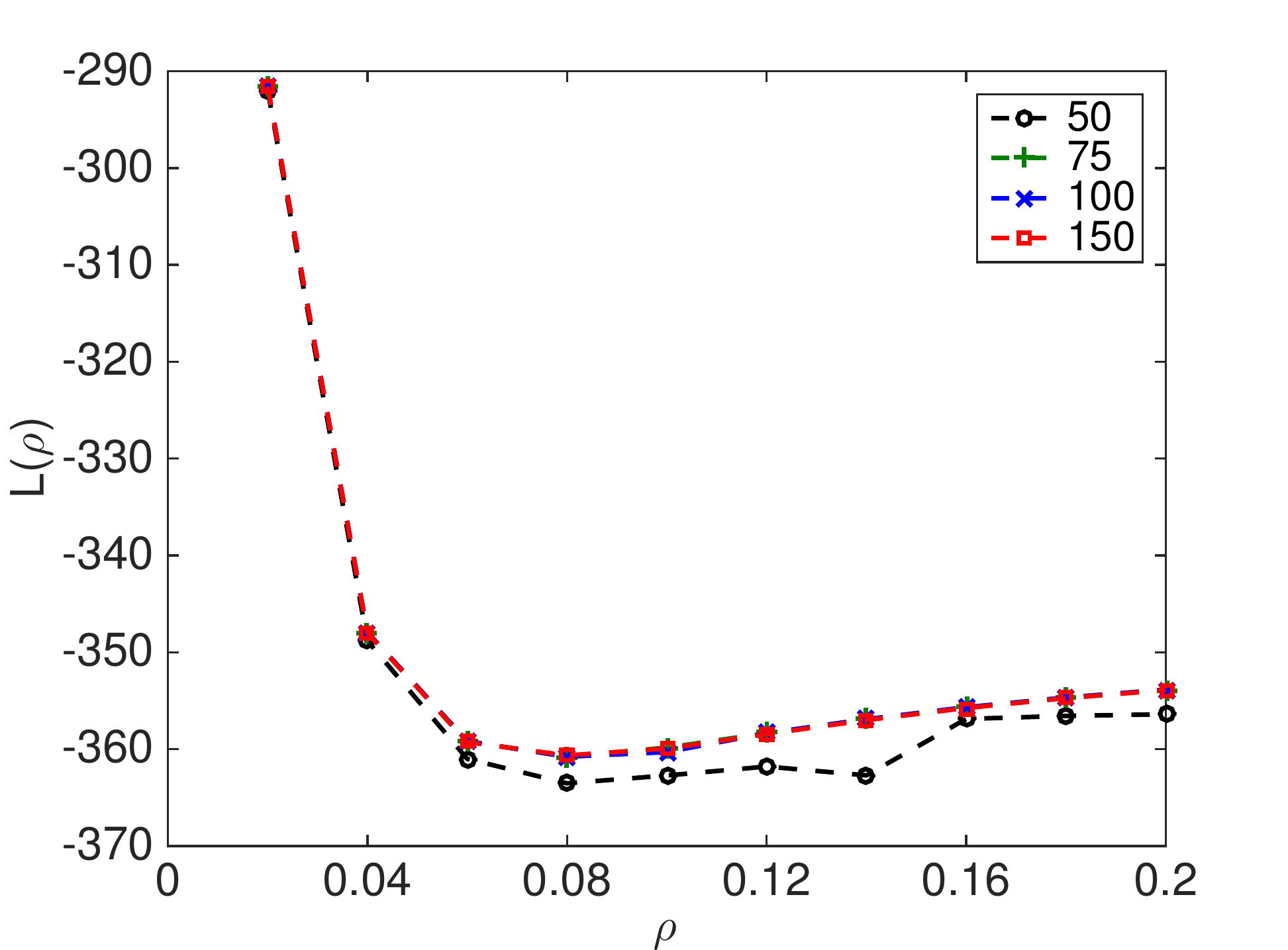}} 
\caption{The negative log marginal posterior function plotted against $\rho$, for a series of ranks $r$ indicated in the legend. Left: $t=0.002$; right: $t=0.02$.}\label{f:L_rho}
\end{figure}

\begin{figure}
\centerline{
\includegraphics[width=.49\textwidth]{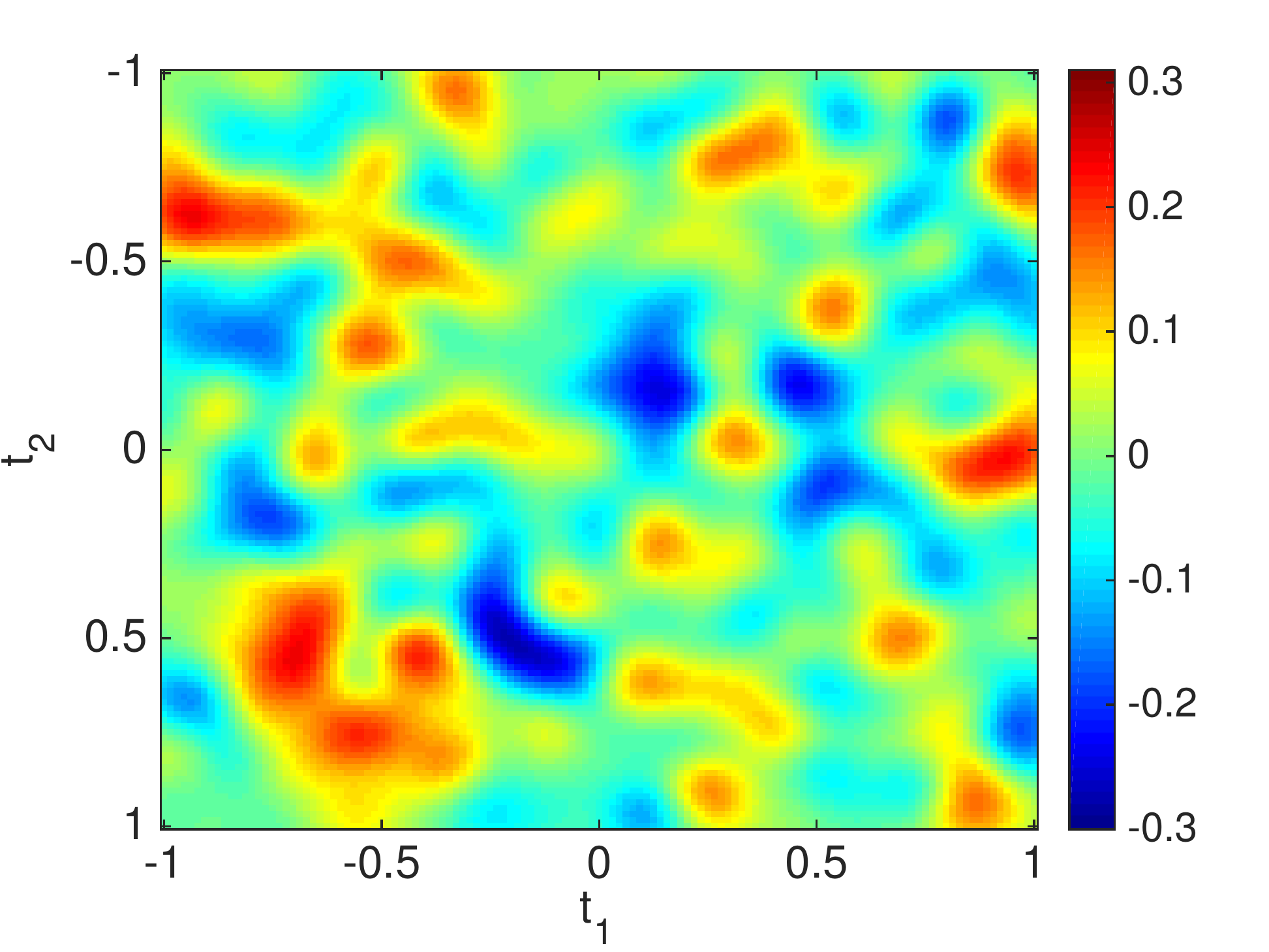}
\includegraphics[width=.49\textwidth]{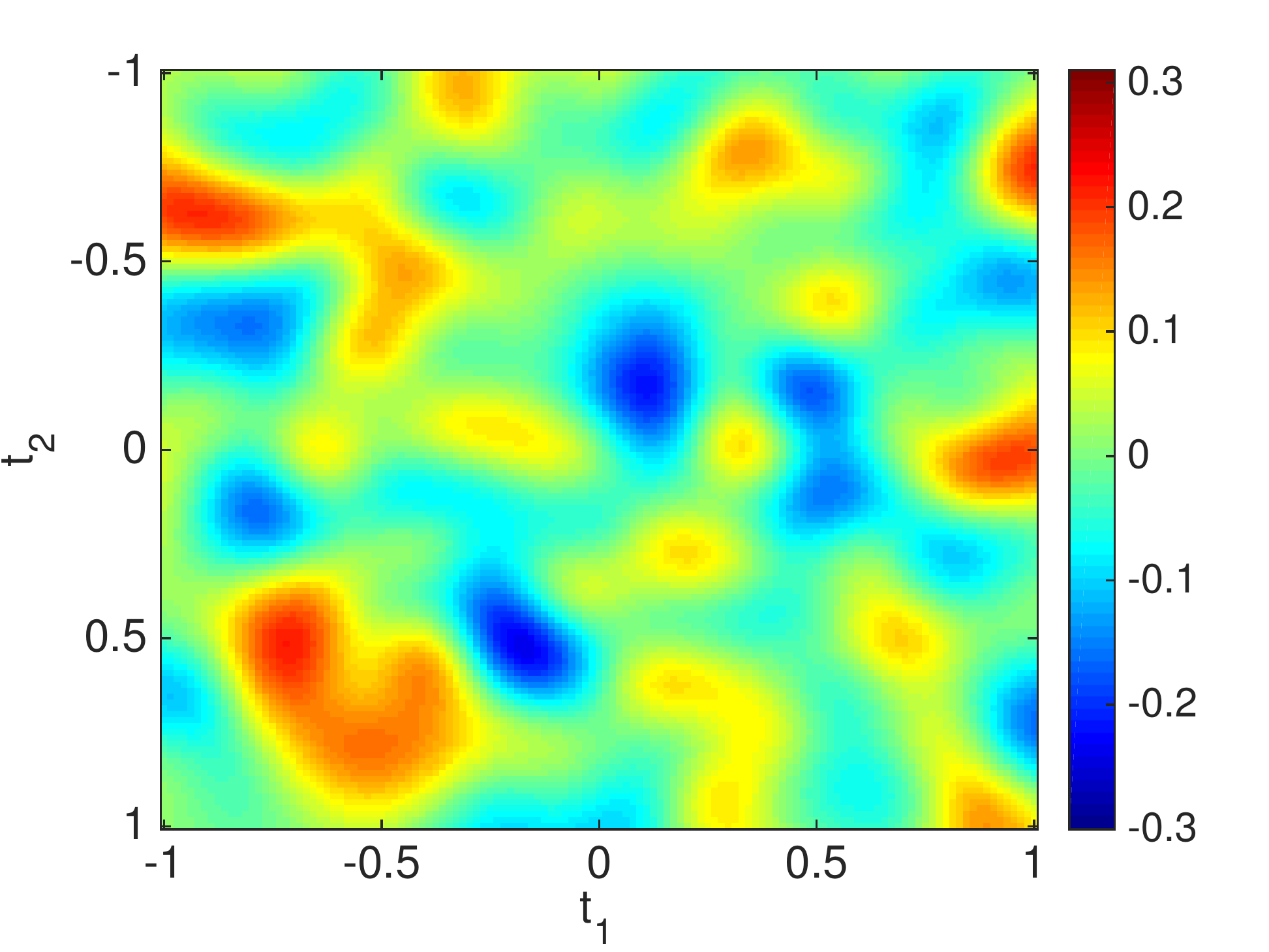}} 
\caption{Posterior mean for the $t=0.002$ case. Left: result for $\rho=0.1$ (optimal); right: result for $\rho=0.2$.}\label{f:postmean_conv}
\end{figure}

\begin{figure}
\centerline{
\includegraphics[width=.49\textwidth]{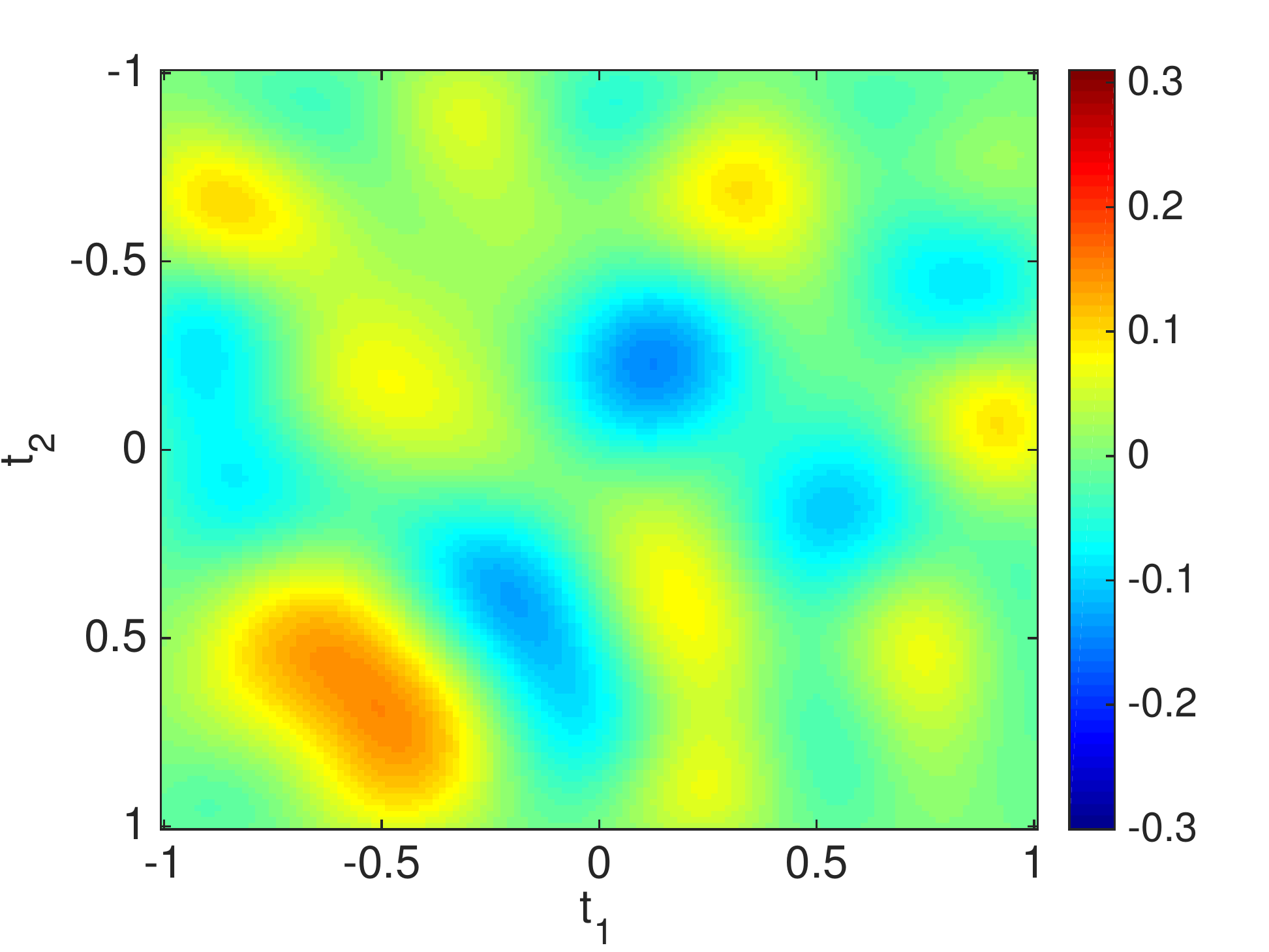}
\includegraphics[width=.49\textwidth]{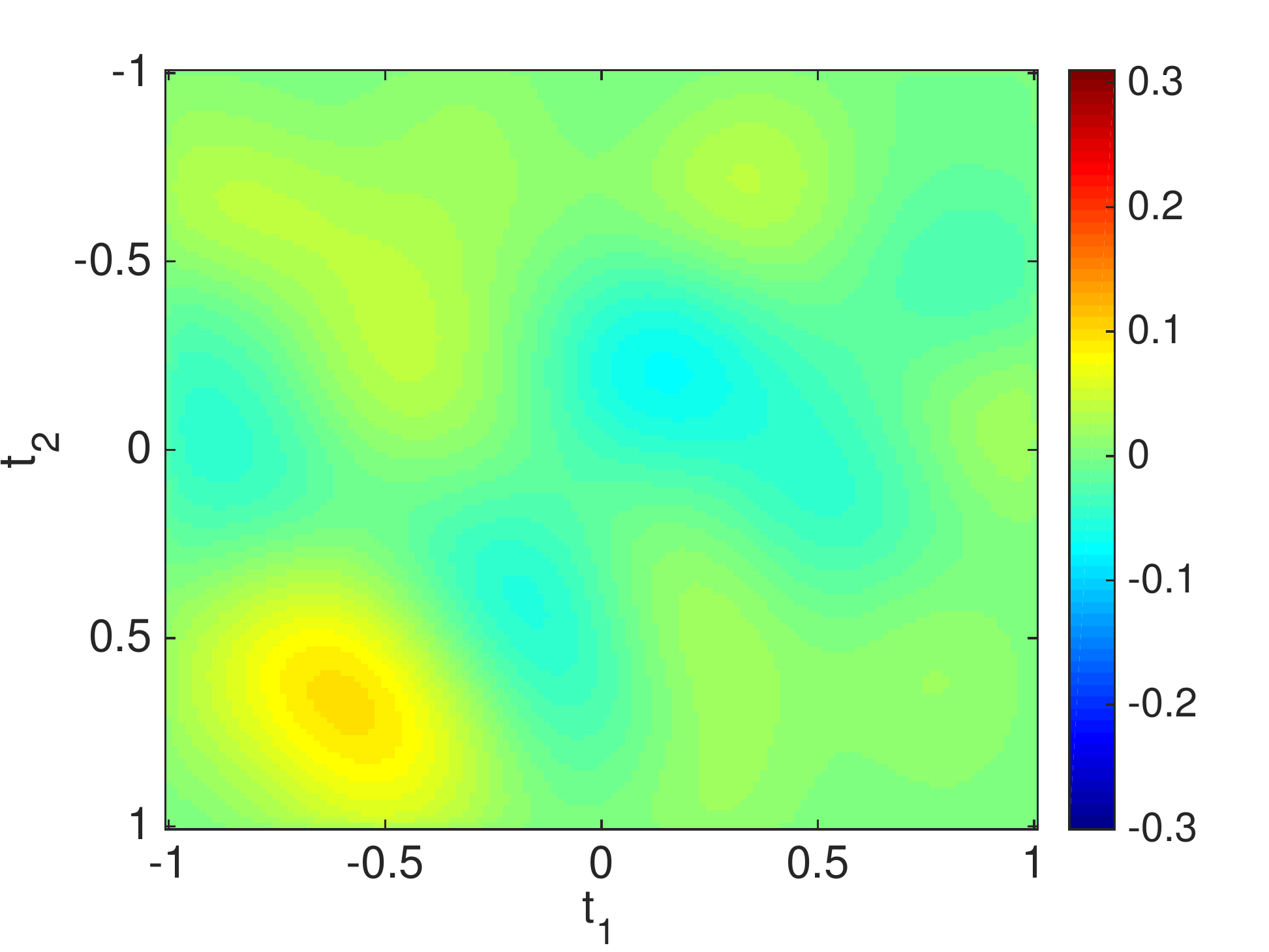}} 
\caption{Posterior mean for the $t=0.02$ case. Left: result for $\rho=0.1$ (optimal); right: result for $\rho=0.02$.}\label{f:postmean2_conv}
\end{figure}

Note that the EB method is able to find the appropriate hyperparameter
values in this example.  Nevertheless, our intention is not to
illustrate that the EB method can always identify the correct value of
the hyperparameters; rather, the main purpose of the example is to
demonstrate that, should one \textit{choose} to use the EB method, the
proposed approximations can efficiently and accurately compute the
marginal likelihood.

\subsection{X-ray computed tomography} 
Our second example is an X-ray computed tomography (CT) problem. 
X-ray CT is a  popular diagnostic tool for medical tomographic imaging of the human body.  It provides images of the object by assigning an X-ray attenuation coefficient to each pixel~\cite{natterer2001mathematics}.  Specifically let $u$ denote the unknown image supported in the unit ball $B(\0,1)$ in $\R^2$.  In the case of two-dimensional parallel beam CT, the projection data~(or sinogram) $f$ for each $\varphi\in[0,2\pi)$ and $s\in\R$ is obtained via the following Radon transform~\cite{Radon1917}:
\begin{equation}\label{Radon_Transform}
f(\varphi,s)=\int_{-\infty}^{\infty}u(s\bth+t\bth^{\perp})dt
\end{equation}
where $\bth=(\cos\varphi,\sin\varphi)$ and $\bth^{\perp}=(-\sin\varphi,\cos\varphi)$.   
{  We test the problem using a $256\times 256$ ground truth image shown in 
Fig.~\ref{f:data_ct} (left), which is taken  from the Harvard Whole Brain Atlas~\cite{brain}.}
The projection data, shown in Fig.~\ref{f:data_ct} (right), is simulated by plugging the true image into the Radon transform and adding measurement noise to the result. 
The noise is taken to be  Gaussian with zero mean and a signal-to-noise ratio of $1\%$.  
The size of the discrete Radon transform operator depends on both the size of the image and the number of projections.  In our simulation, we used $60$ projections equispatially sampled from $0$ to $\pi$. 
{  
To avoid an inverse crime~\cite{kaipio2007statistical}, we recover the image with a $128\times128$ resolution, and  thus the dimensionality of the inverse problem is $128^2$.  
We take the prior distribution to be zero-mean Gaussian, with a covariance matrix given by an anisotropic Mat\'{e}rn kernel:
\begin{equation}
K(\-t,\-t') = \sigma^2 \frac{2^{1-\nu}}{\Gamma(\nu)} \left ( \sqrt{2\nu}d(\-t,\-t') \right )^\nu B_\nu \left ( \sqrt{2\nu} d(\-t,\-t') \right ), \label{e:animatern}
\end{equation}
where  
\[
d(\-t,\-t') =\sqrt{ \frac{(t_1-t'_1)^2}{\rho_1^2}+\frac{(t_2-t'_2)^2}{\rho_2^2}}.
\]
 Now we consider all four parameters $\nu$, $\sigma$, $\rho_1$ and $\rho_2$ as hyperparameters to be determined,
 and we also assume that the variance $\epsilon^2$ of the measurement noise is not available. 
We will identify the five parameters using the proposed low-rank approximate EB method.}

We solve the optimization problems with 9 different ranks ($r=500$, 1000, 1500, 2000, 2500, 3000, 3500, 4000, 4500), and 
plot the optimal values of the five hyperparameters and the associated value of $L$ versus the rank $r$ in Fig.~\ref{f:opt_r}.   
As one can see, the obtained optimal values of the hyperparameters converge as the rank increases, and in particular 
$r=4000$ is sufficient to obtain accurate estimates of  the hyperparameters. 
Next we shall illustrate that the obtained hyperparameter values can lead to rather good inference results. 
To show this, we compute the posterior mean and the variance of the image using the hyperparameter values computed with the various ranks,
and we then compute the peak noise-to-signal ratio (PSNR) of the posterior means against the true
image, which is a commonly used measure of quality for image reconstruction.
{ Note that the low rank approximation is only used to obtain the 
hyperparameters; given these hyperparameter values, the posterior means and variances are computed directly without approximation.}
We plot the PSNR as a function of the rank $r$ in Fig.~\ref{f:psnr}. 
We also show the posterior mean and variance field associated with each rank in the plot.  One can see from the figure that, as the rank increases, the resulting hyperparameters can recover images with higher quality in terms of PSNR.  
Finally, in Fig.~\ref{f:recovery}, we plot the posterior mean and variance using the hyperparameters computed with $r=4500$, 
which are taken to be the inference results of this problem.

These results demonstrate that the outcome of Bayesian inference can depend critically on the
hyperparameters, and that, at least in this example, the proposed method can
provide reasonable values for these parameters, provided that a sufficiently large rank is used.
We also note here that,  to better recover images with sharp edges,
one might resort to {more complex Gaussian
  hypermodels}~\cite{bardsley2010hierarchical,calvetti2007gaussian,calvetti2008hypermodels}
than the one considered here, or even non-Gaussian models such as the
Besov~\cite{dashti2012besov,lassas2009discretization} and TV-Gaussian~\cite{yao2016tv} priors.
Nevertheless, if one has chosen to use a Gaussian hypermodel, our method can efficiently
determine the values of the associated hyperparameters. 

\begin{figure}
\centerline{\includegraphics[width=.5\textwidth]{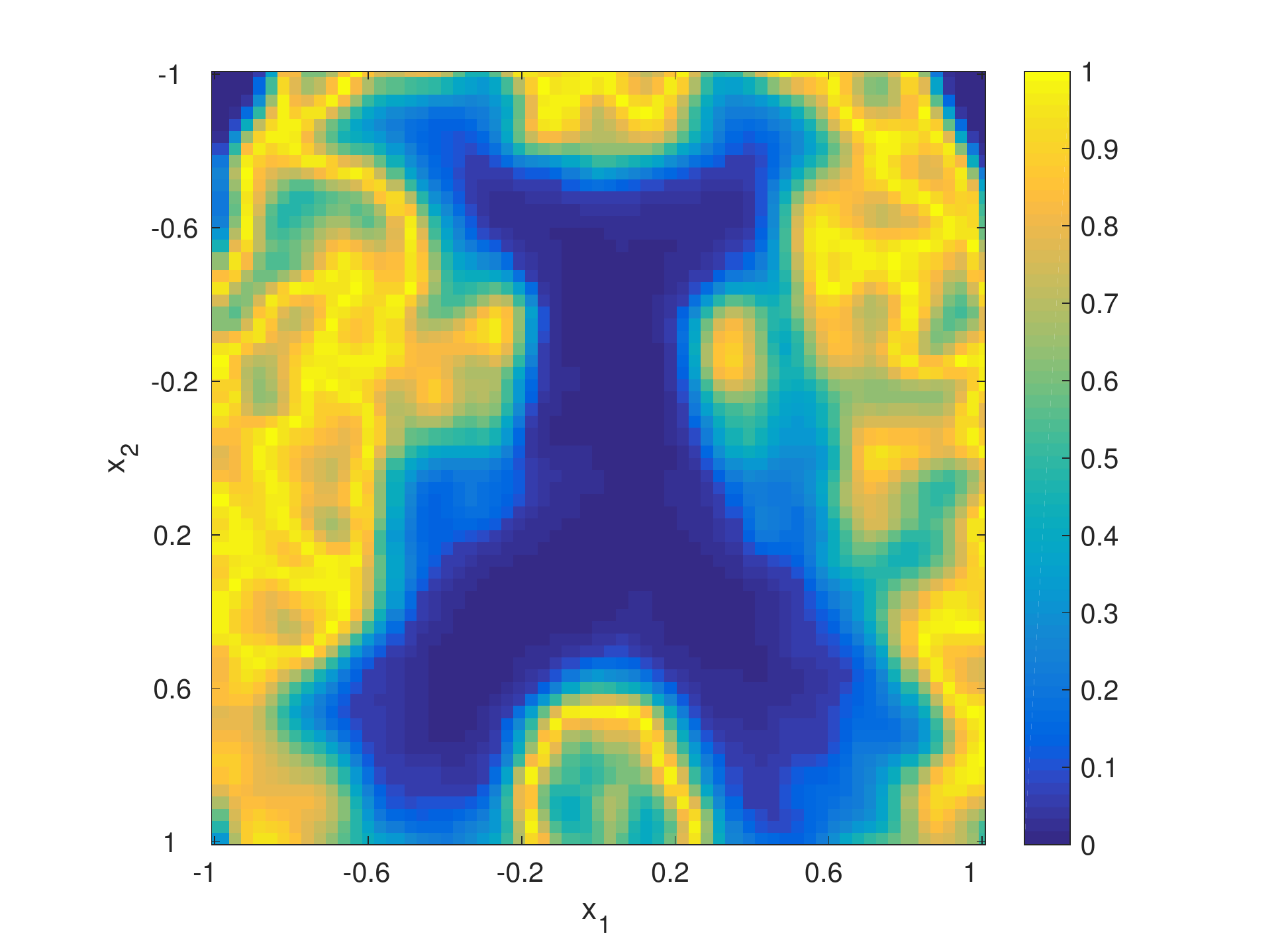}
\includegraphics[width=.5\textwidth]{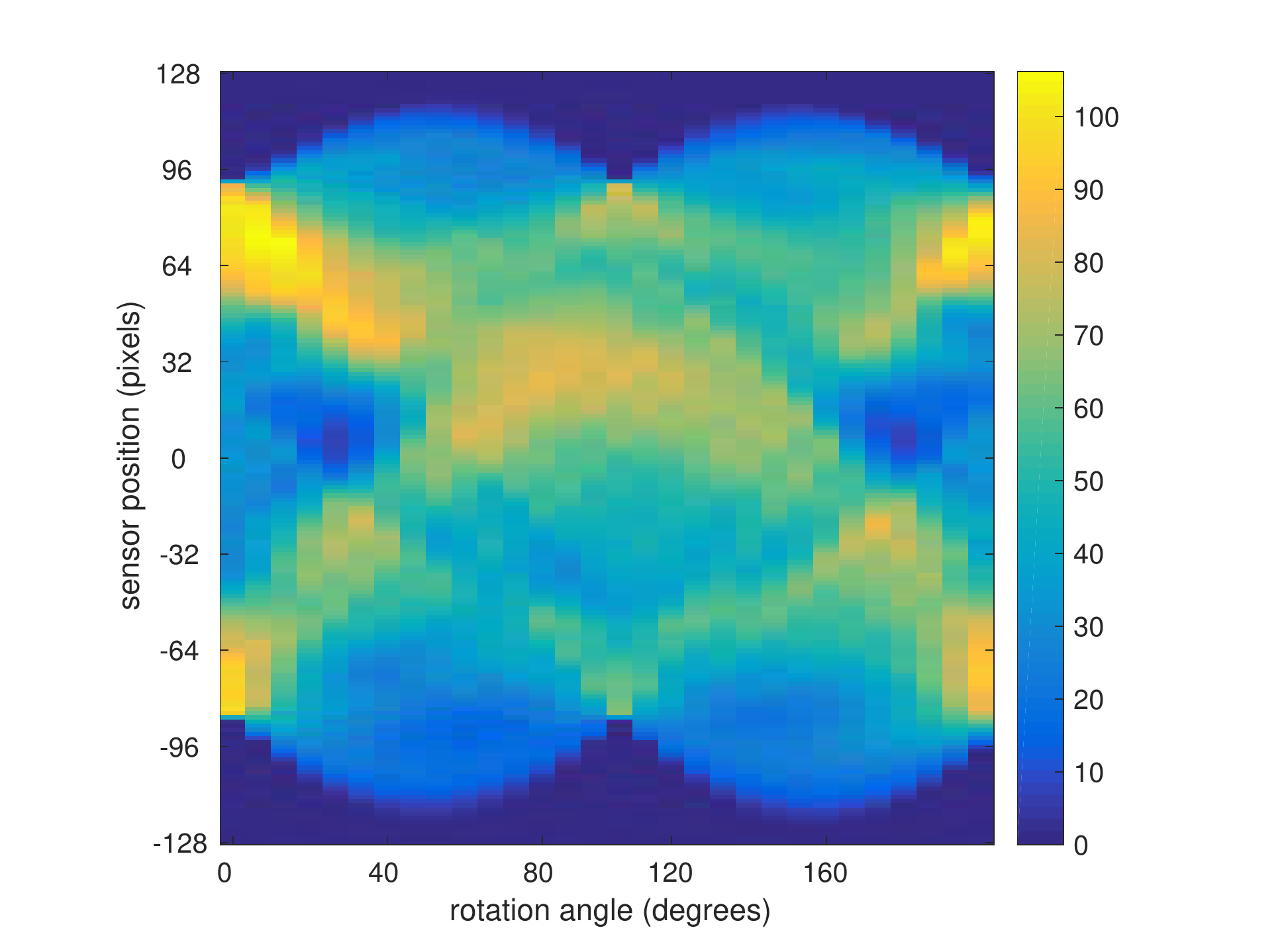}} 
\caption{Left: the ground truth image. Right:  the projection data simulated from the ground truth.}\label{f:data_ct}
\end{figure}

\begin{figure}
\leftline{\includegraphics[width=.33\textwidth]{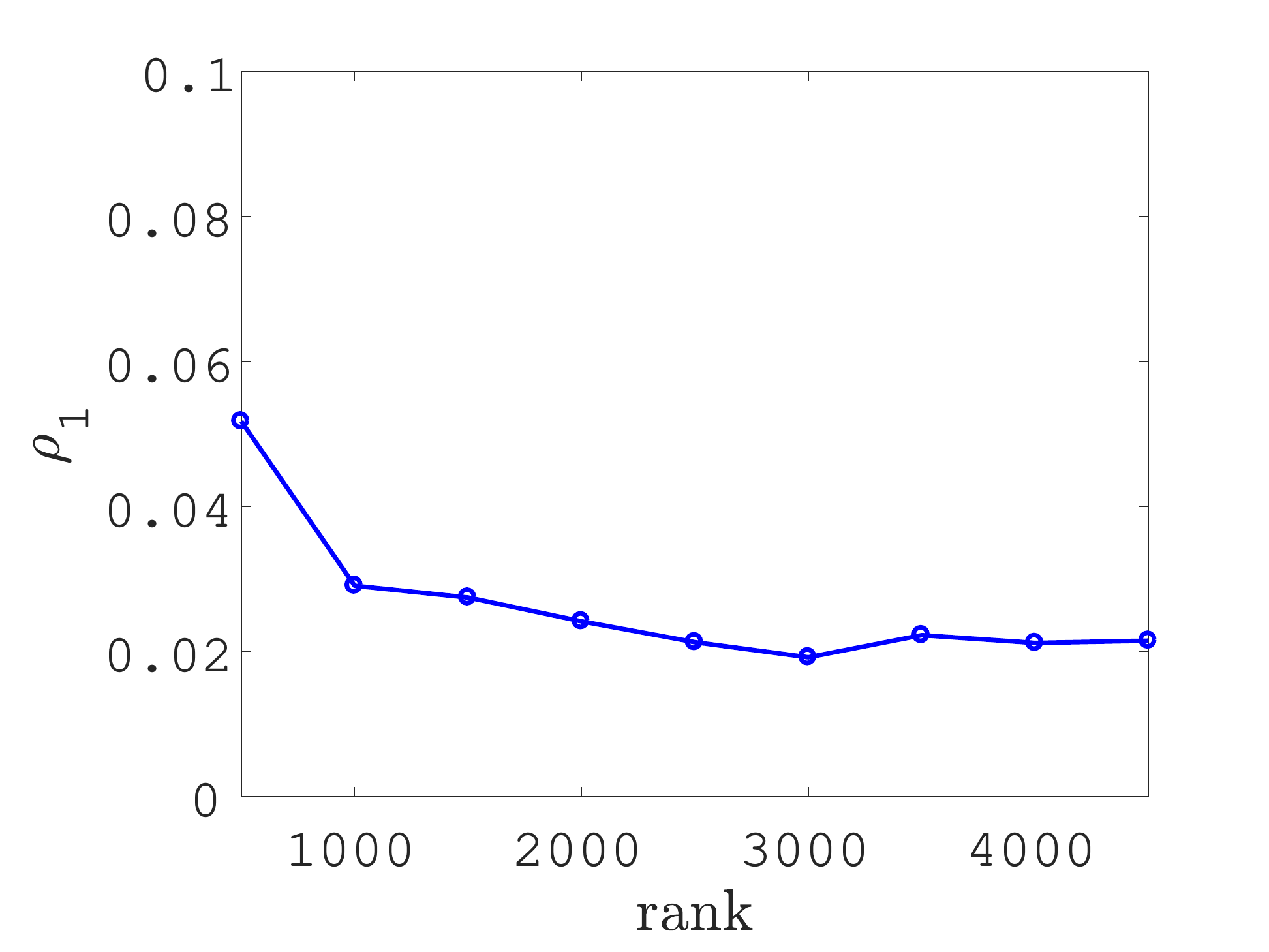}\includegraphics[width=.33\textwidth]{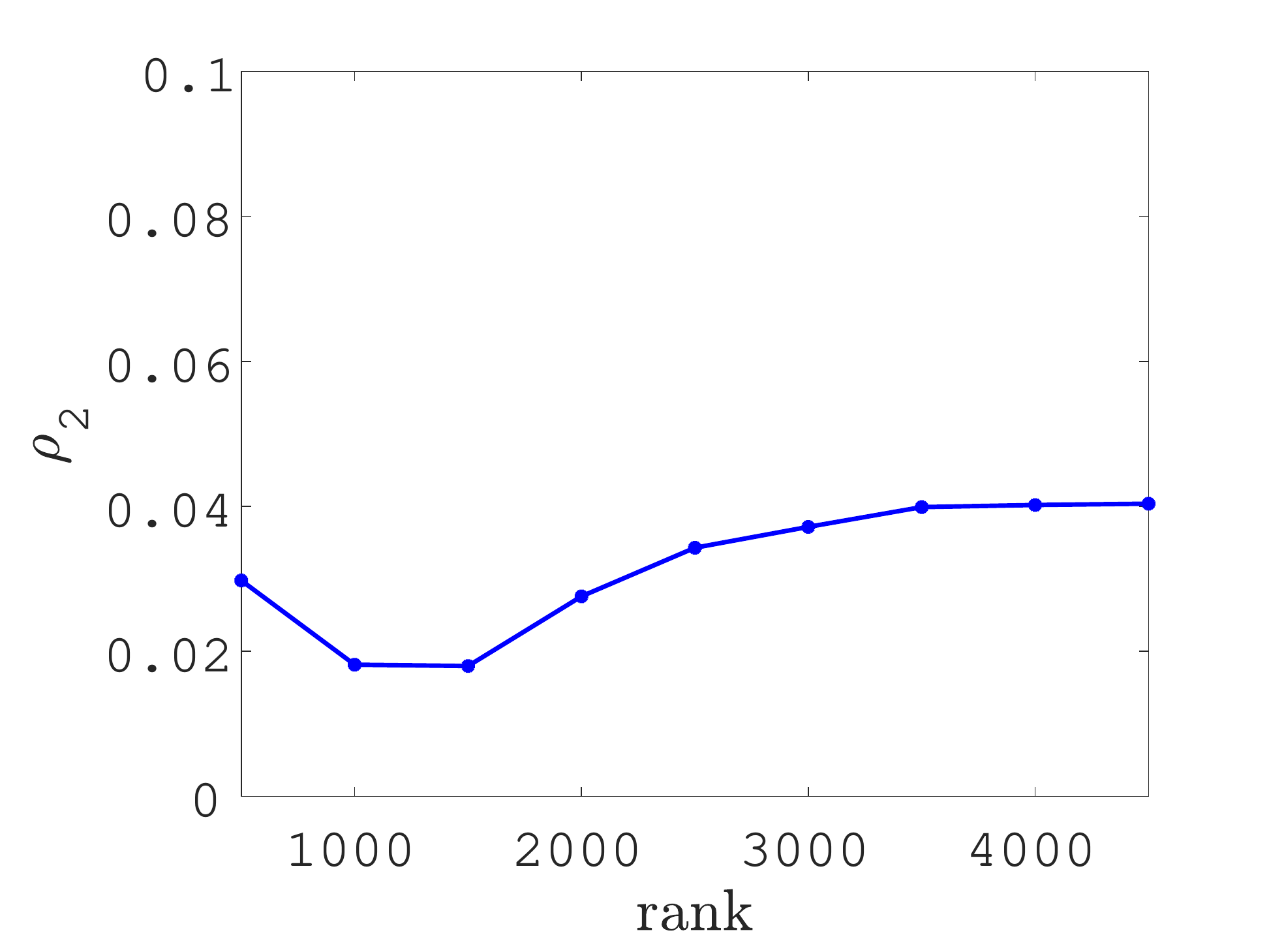}
\includegraphics[width=.33\textwidth]{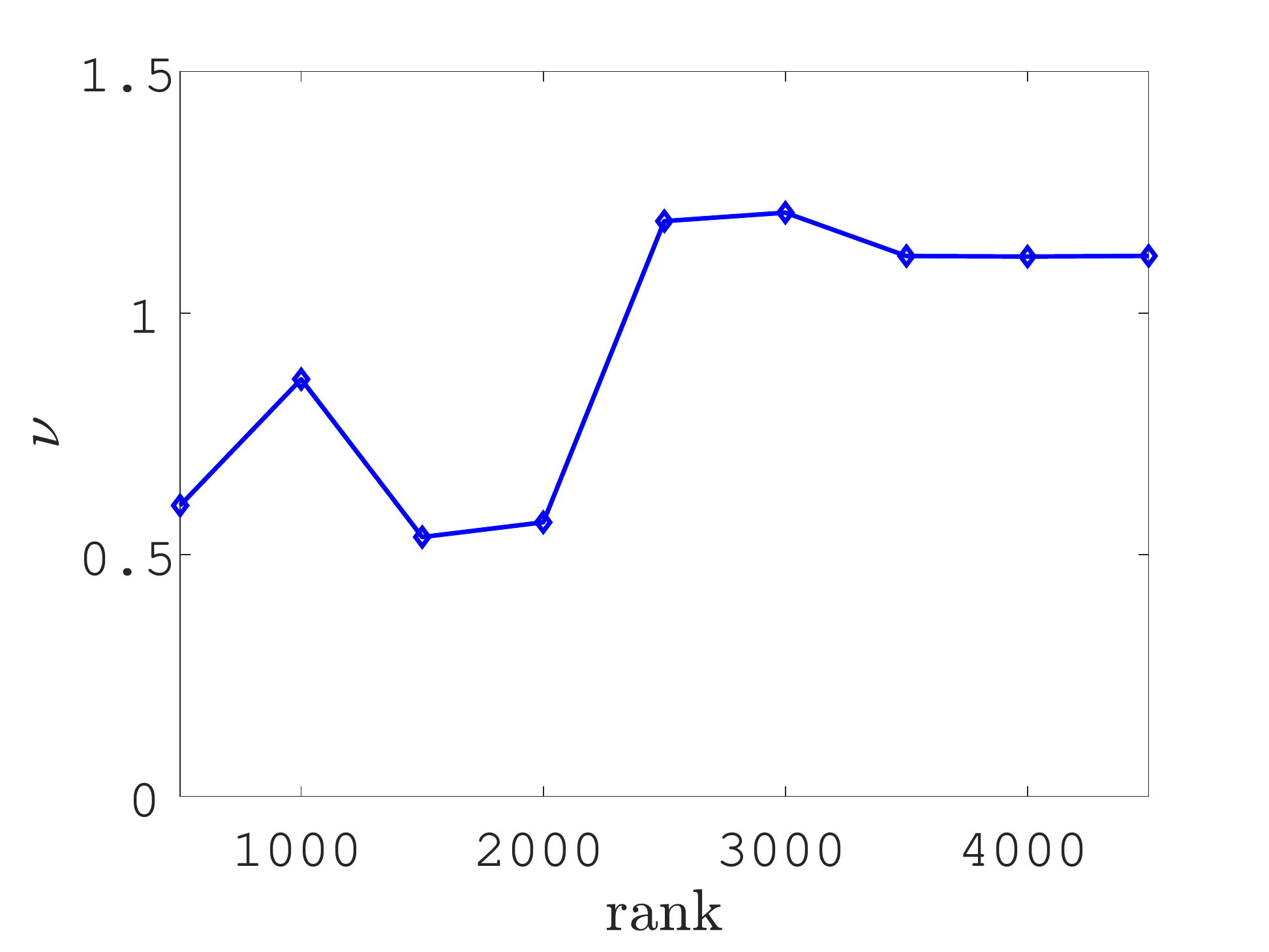}} 
\leftline{\includegraphics[width=.33\textwidth]{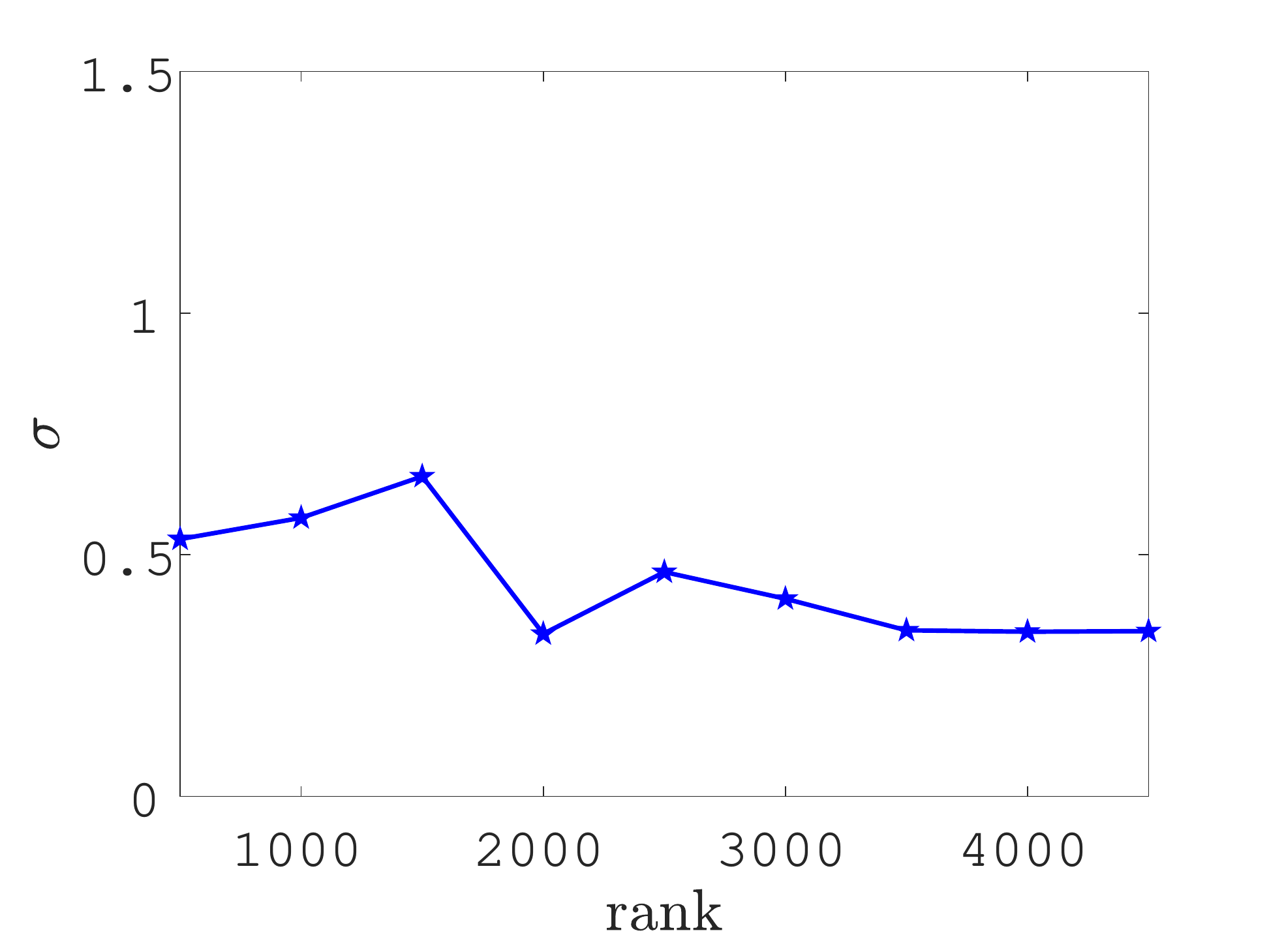}
\includegraphics[width=.33\textwidth]{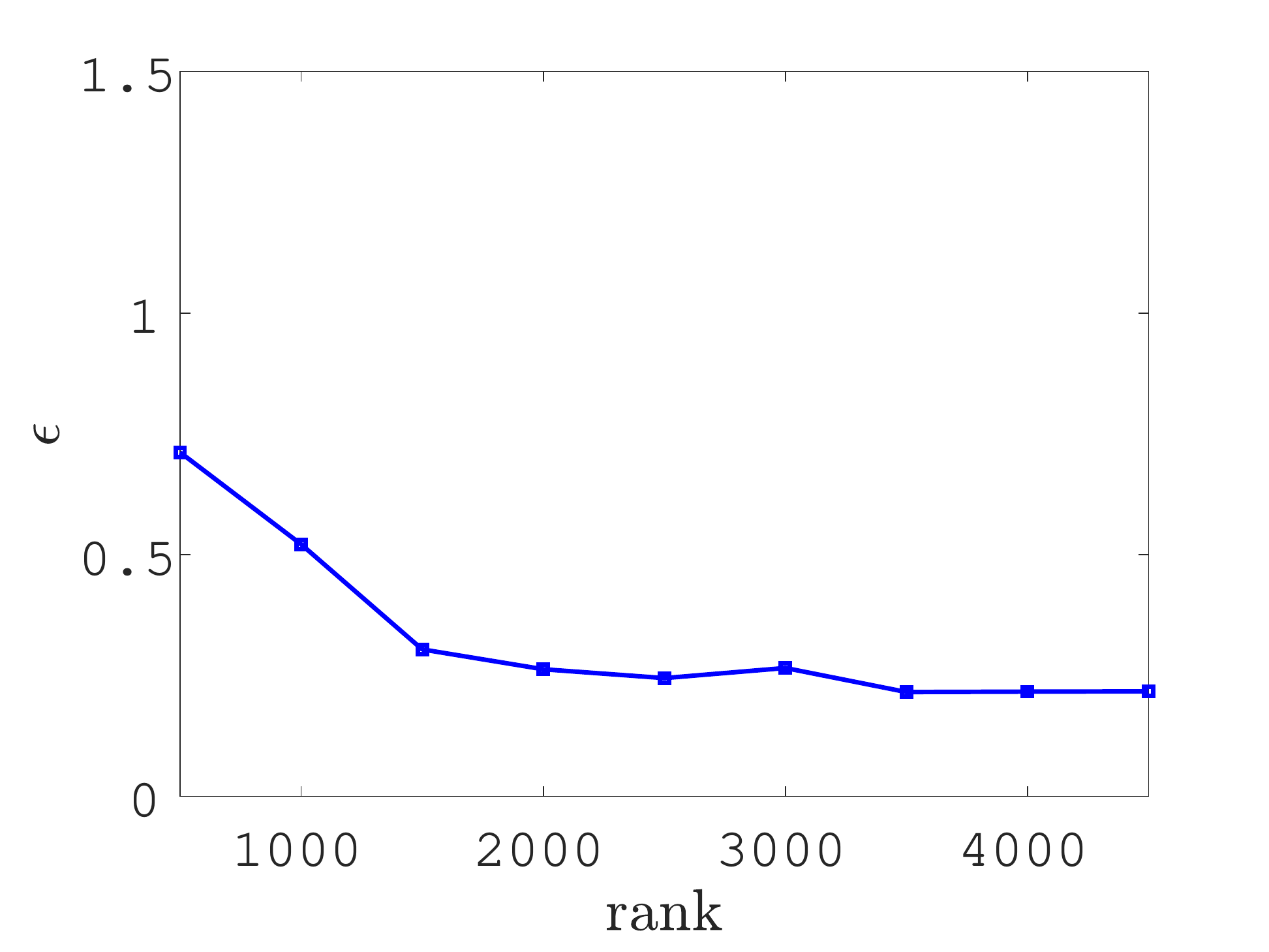}
\includegraphics[width=.33\textwidth]{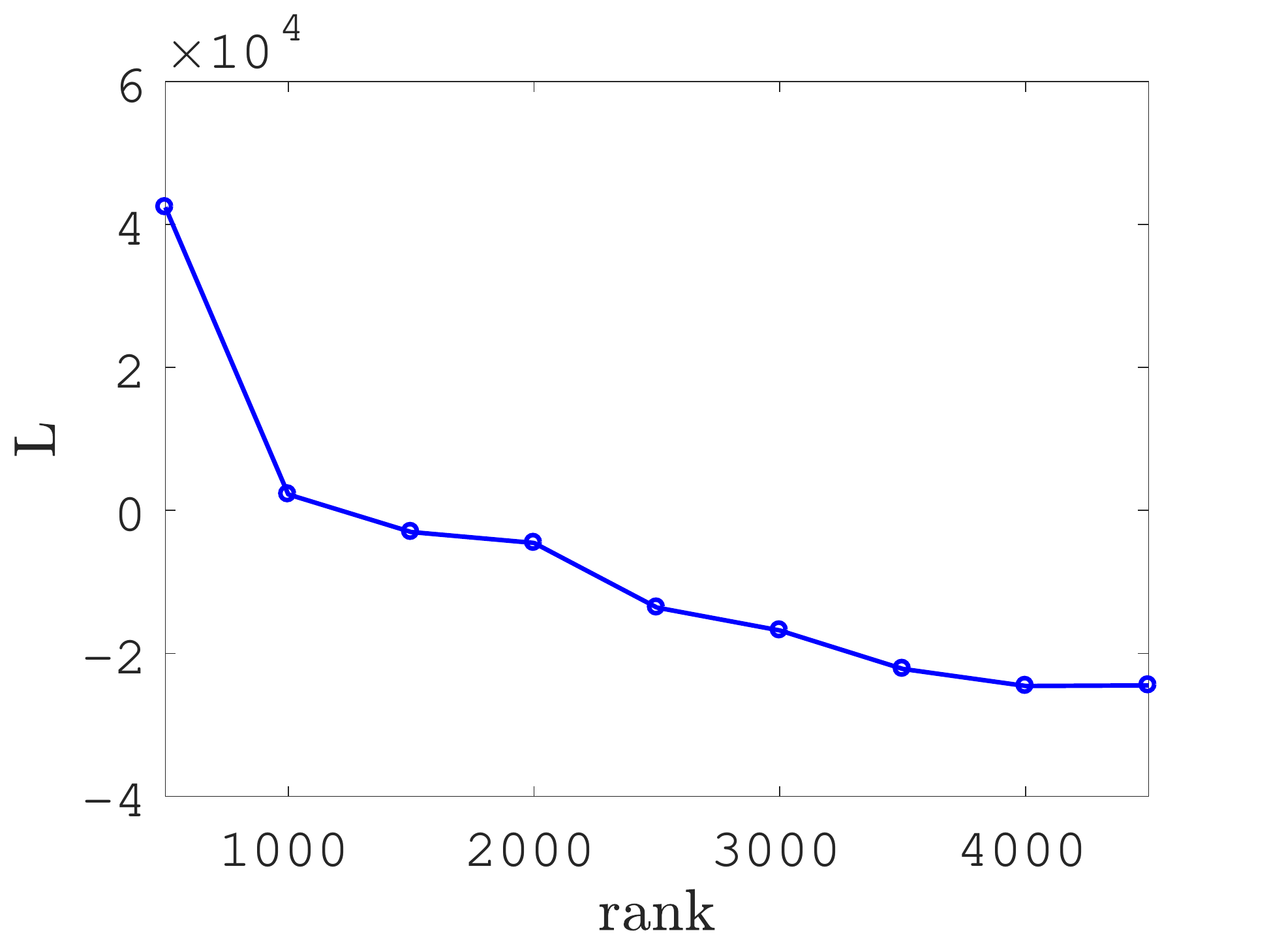}} 
\caption{The optimal value of the five hyperparameters and the associated value of $L$ (bottom, right) plotted against the rank $r$ used in
  the EB approximation. }\label{f:opt_r}
\end{figure}

\begin{figure}
\centerline{\includegraphics[width=.95\textwidth]{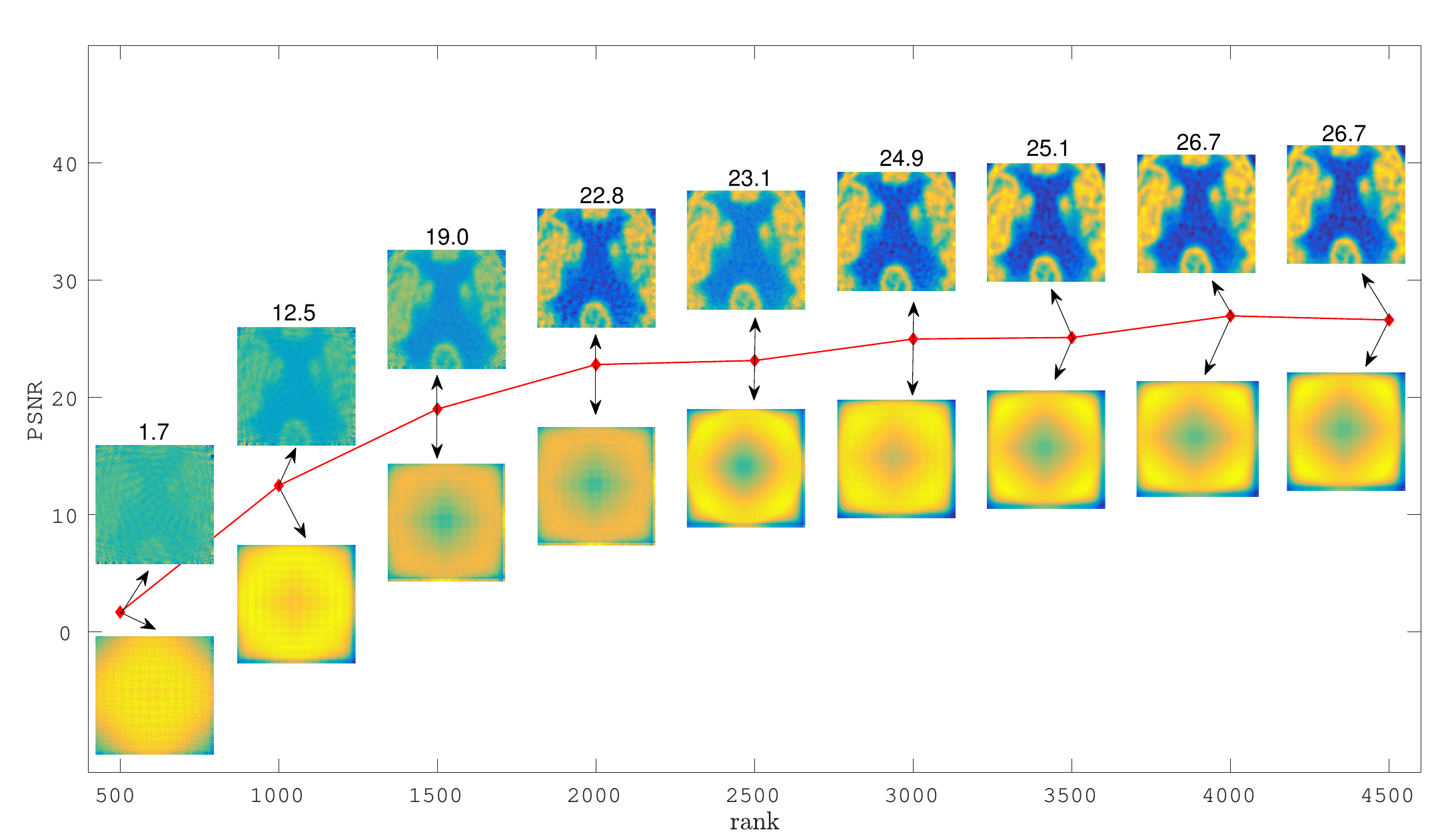}}
\caption{The PSNR of the posterior mean versus the rank $r$ (red line). Above and below the line, we also show the posterior mean 
and variance obtained with the hyperparameters computed with each rank.}\label{f:psnr}
\end{figure}

\begin{figure}
\centerline{\includegraphics[width=.49\textwidth]{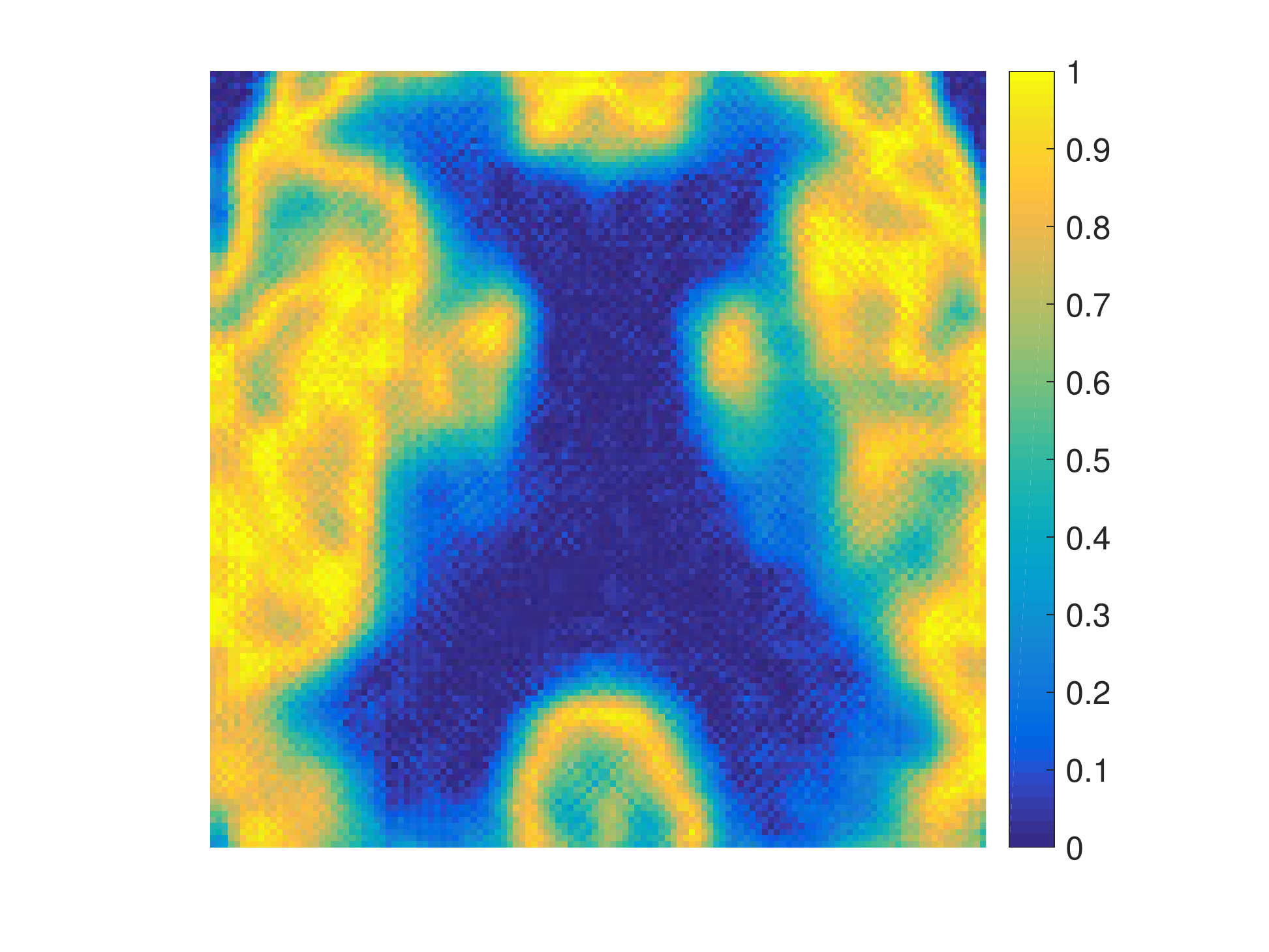}
\includegraphics[width=.49\textwidth]{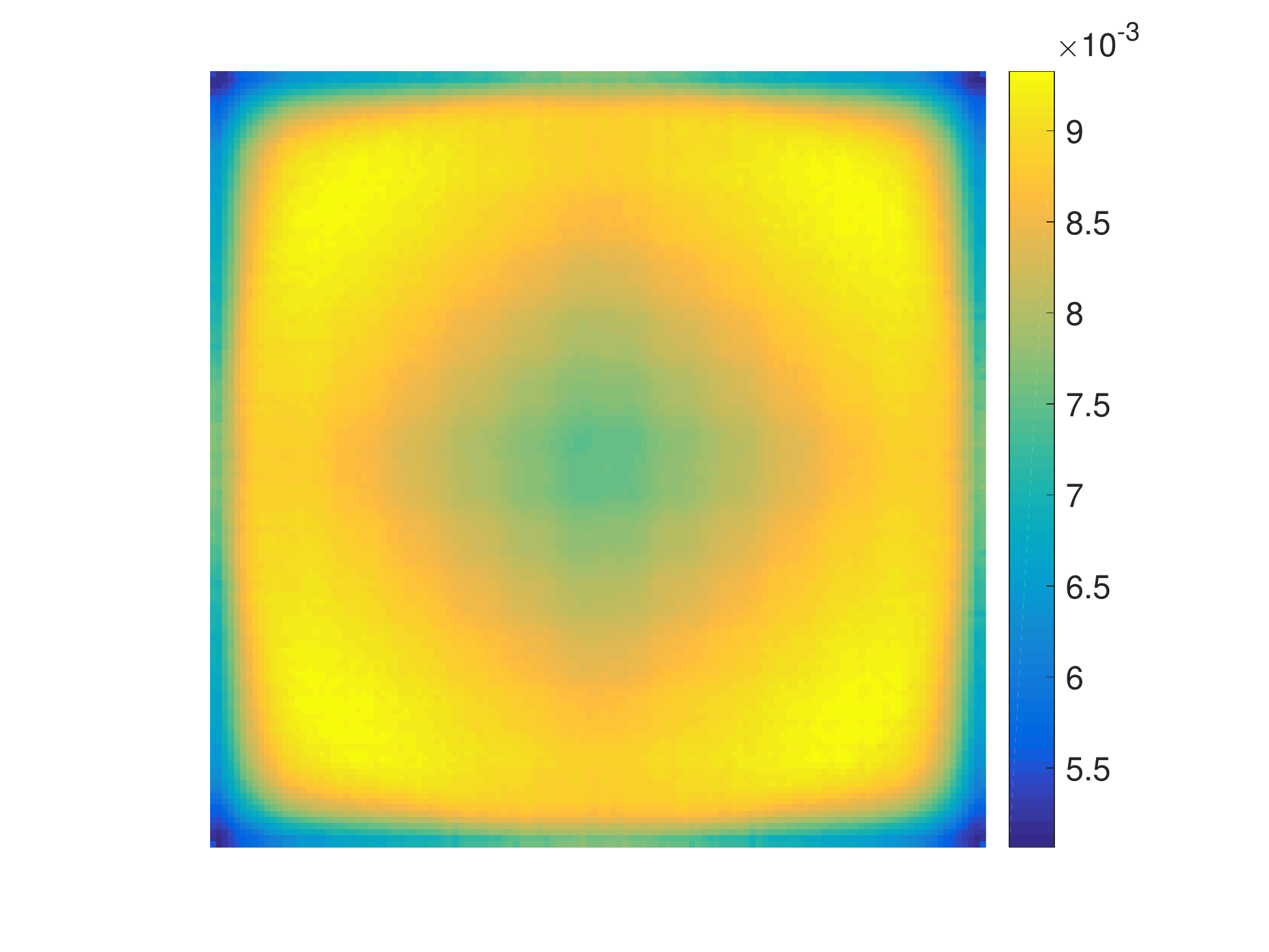}}  
\caption{The posterior mean and  the posterior variance  computed using the optimized hyperparameter values with $r=4500$.}\label{f:recovery}
\end{figure}

\section{Conclusions}\label{sec:conclusions}

This paper investigates empirical Bayesian methods for determining hyperparameters in linear inverse problems.
%
We develop an efficient numerical method to approximately evaluate the marginal likelihood
function of the hyperparameters for large-scale problems, based on a low-rank
approximation of the update from the prior to the posterior covariance. 
The proposed method can achieve a computational complexity of $O(n^2 r)$ for $r\ll n$, while 
a standard full-rank or direct evaluation approach requires computations of $O(n^3)$ complexity. 
We also show that
this approximation of the marginal likelihood is optimal in a minimax sense.
Two numerical examples illustrate that the proposed algorithm can accurately evaluate the marginal likelihood function
required  in the EB method. 
This approach may be useful in a wide range of applications where the
unknown is of very high dimension, such as medical and geophysical image
reconstruction.
{  Finally, it is worth noting that, while the approximate EB method presented in this work can only be applied to linear inverse problems, there have been several efforts successfully  extending low-rank posterior approximations to nonlinear problems~\cite{martin2012stochastic,cui2014likelihood}; along these lines, we expect that our method could also be extended to EB computations for nonlinear inverse problems.}

\section*{Acknowledgement}
Li was partially supported by the NSFC under grant number 11771289. Marzouk was
partially supported by the US Department of Energy, Office of Advanced Scientific
Computing Research, under grant number DE-SC0009297.

\medskip

\bibliographystyle{plain}
\bibliography{lowrank}

\begin{thebibliography}{10}

\bibitem{alexanderian2016bayesian}
Alen Alexanderian, Philip~J Gloor, Omar Ghattas, et~al.
\newblock On {B}ayesian {A-and D-optimal} experimental designs in infinite
  dimensions.
\newblock {\em Bayesian Analysis}, 11(3):671--695, 2016.

\bibitem{bardsley2010hierarchical}
Johnathan~M Bardsley, Daniela Calvetti, and Erkki Somersalo.
\newblock Hierarchical regularization for edge-preserving reconstruction of
  {PET} images.
\newblock {\em Inverse Problems}, 26(3):035010, 2010.

\bibitem{bui2013computational}
Tan Bui-Thanh, Omar Ghattas, James Martin, and Georg Stadler.
\newblock A computational framework for infinite-dimensional {B}ayesian inverse
  problems part i: The linearized case, with application to global seismic
  inversion.
\newblock {\em SIAM Journal on Scientific Computing}, 35(6):A2494--A2523, 2013.

\bibitem{calvetti2007gaussian}
Daniela Calvetti and Erkki Somersalo.
\newblock A {G}aussian hypermodel to recover blocky objects.
\newblock {\em Inverse problems}, 23(2):733, 2007.

\bibitem{calvetti2008hypermodels}
Daniela Calvetti and Erkki Somersalo.
\newblock Hypermodels in the {Bayesian} imaging framework.
\newblock {\em Inverse Problems}, 24(3):034013, 2008.

\bibitem{carlin2000bayes}
Bradley~P Carlin and Thomas~A Louis.
\newblock {\em {Bayes and empirical Bayes methods for data analysis}}.
\newblock Chapman \& Hall/CRC Boca Raton, 2000.

\bibitem{casella1985introduction}
George Casella.
\newblock An introduction to empirical {B}ayes data analysis.
\newblock {\em The American Statistician}, 39(2):83--87, 1985.

\bibitem{cui2014likelihood}
Tiangang Cui, James Martin, Youssef~M Marzouk, Antti Solonen, and Alessio
  Spantini.
\newblock Likelihood-informed dimension reduction for nonlinear inverse
  problems.
\newblock {\em Inverse Problems}, 30(11):114015, 2014.

\bibitem{dashti2012besov}
Masoumeh Dashti, Stephen Harris, and Andrew Stuart.
\newblock Besov priors for {Bayesian} inverse problems.
\newblock {\em Inverse Problems and Imaging}, 6(2):183--200, 2012.

\bibitem{drineas2006fast}
Petros Drineas, Ravi Kannan, and Michael~W Mahoney.
\newblock Fast {Monte Carlo} algorithms for matrices ii: Computing a low-rank
  approximation to a matrix.
\newblock {\em SIAM Journal on Computing}, 36(1):158--183, 2006.

\bibitem{flath2011fast}
H~Pearl Flath, Lucas~C Wilcox, Volkan Ak{\c{c}}elik, Judith Hill, Bart van
  Bloemen~Waanders, and Omar Ghattas.
\newblock Fast algorithms for {Bayesian} uncertainty quantification in
  large-scale linear inverse problems based on low-rank partial hessian
  approximations.
\newblock {\em SIAM Journal on Scientific Computing}, 33(1):407--432, 2011.

\bibitem{fornberg1998practical}
Bengt Fornberg.
\newblock {\em A practical guide to pseudospectral methods}, volume~1.
\newblock Cambridge university press, 1998.

\bibitem{halko2011finding}
Nathan Halko, Per-Gunnar Martinsson, and Joel~A Tropp.
\newblock Finding structure with randomness: Probabilistic algorithms for
  constructing approximate matrix decompositions.
\newblock {\em SIAM review}, 53(2):217--288, 2011.

\bibitem{hansen2006deblurring}
Per~Christian Hansen, James~G Nagy, and Dianne~P O'leary.
\newblock {\em Deblurring images: matrices, spectra, and filtering}.
\newblock SIAM, 2006.

\bibitem{jiang2013fast}
Shidong Jiang, Zhi Liang, and Jingfang Huang.
\newblock A fast algorithm for {B}rownian dynamics simulation with hydrodynamic
  interactions.
\newblock {\em Mathematics of Computation}, 82(283):1631--1645, 2013.

\bibitem{brain}
Keith~A. Johnson and J.~Alex Becker.
\newblock Whole {B}rain {A}tlas, http://www.med.harvard.edu/aanlib/.

\bibitem{kaipio2005statistical}
Jari Kaipio and Erkki Somersalo.
\newblock {\em Statistical and computational inverse problems}, volume 160.
\newblock Springer, 2005.

\bibitem{kaipio2007statistical}
Jari Kaipio and Erkki Somersalo.
\newblock Statistical inverse problems: discretization, model reduction and
  inverse crimes.
\newblock {\em Journal of computational and applied mathematics},
  198(2):493--504, 2007.

\bibitem{lassas2009discretization}
Matti Lassas, Eero Saksman, and Samuli Siltanen.
\newblock Discretization-invariant {Bayesian} inversion and besov space priors.
\newblock {\em Inverse Problems and Imaging}, 3(1):87--122, 2009.

\bibitem{liberty2007randomized}
Edo Liberty, Franco Woolfe, Per-Gunnar Martinsson, Vladimir Rokhlin, and Mark
  Tygert.
\newblock Randomized algorithms for the low-rank approximation of matrices.
\newblock {\em Proceedings of the National Academy of Sciences},
  104(51):20167--20172, 2007.

\bibitem{martin2012stochastic}
James Martin, Lucas~C Wilcox, Carsten Burstedde, and Omar Ghattas.
\newblock A stochastic {Newton MCMC} method for large-scale statistical inverse
  problems with application to seismic inversion.
\newblock {\em SIAM Journal on Scientific Computing}, 34(3):A1460--A1487, 2012.

\bibitem{natterer2001mathematics}
Frank Natterer.
\newblock {\em The mathematics of computerized tomography}.
\newblock SIAM, 2001.

\bibitem{petrone2014bayes}
Sonia Petrone, Judith Rousseau, and Catia Scricciolo.
\newblock Bayes and empirical {B}ayes: do they merge?
\newblock {\em Biometrika}, page ast067, 2014.

\bibitem{press2007numerical}
William~H Press.
\newblock {\em Numerical recipes 3rd edition: The art of scientific computing}.
\newblock Cambridge university press, 2007.

\bibitem{Radon1917}
Johann Radon.
\newblock {{\"U}ber die Bestimmung von Funktionen durch ihre Integralwerte
  l{\"a}ngs gewisser Mannigfaltigkeiten}.
\newblock {\em Mathematisch-Physische Klasse}, 69:262--277, 1917.

\bibitem{rao1945information}
Calyampudi~Radhakrishna Rao.
\newblock Information and accuracy attainable in the estimation of statistical
  parameters.
\newblock {\em Bull. Calcutta Math. Soc}, 37(3):81--91, 1945.

\bibitem{rasmussen2006gaussian}
Carl~Edward Rasmussen.
\newblock {\em Gaussian processes for machine learning}.
\newblock MIT Press, 2006.

\bibitem{rousseau2015asymptotic}
Judith Rousseau and Botond Szabo.
\newblock Asymptotic behaviour of the empirical bayes posteriors associated to
  maximum marginal likelihood estimator.
\newblock {\em arXiv preprint arXiv:1504.04814}, 2015.

\bibitem{spantini2016goal}
Alessio Spantini, Tiangang Cui, Karen Willcox, Luis Tenorio, and Youssef
  Marzouk.
\newblock Goal-oriented optimal approximations of {Bayesian} linear inverse
  problems.
\newblock {\em arXiv preprint arXiv:1607.01881}, 2016.

\bibitem{spantini2015optimal}
Alessio Spantini, Antti Solonen, Tiangang Cui, James Martin, Luis Tenorio, and
  Youssef Marzouk.
\newblock Optimal low-rank approximations of {Bayesian} linear inverse
  problems.
\newblock {\em SIAM Journal on Scientific Computing}, 37(6):A2451--A2487, 2015.

\bibitem{stein2012interpolation}
Michael~L Stein.
\newblock {\em Interpolation of spatial data: some theory for kriging}.
\newblock Springer Science \& Business Media, 2012.

\bibitem{stuart2010inverse}
Andrew~M. Stuart.
\newblock Inverse problems: a {Bayesian} perspective.
\newblock {\em Acta Numerica}, 19:451--559, 2010.

\bibitem{tarantola2005inverse}
Albert Tarantola.
\newblock {\em Inverse problem theory and methods for model parameter
  estimation}.
\newblock siam, 2005.

\bibitem{yao2016tv}
Zhewei Yao, Zixi Hu, and Jinglai Li.
\newblock A {TV-Gaussian prior for infinite-dimensional Bayesian} inverse
  problems and its numerical implementations.
\newblock {\em Inverse Problems}, 32(7):075006, 2016.

\end{thebibliography}

\end{document}